\documentclass[11pt, reqno]{amsart}
\usepackage[utf8]{inputenc}
\usepackage[T1]{fontenc}
\usepackage{microtype}
\usepackage{amsmath}
\usepackage{amsfonts}
\usepackage[english]{babel}
\usepackage{mathrsfs}
\usepackage{amssymb}
\usepackage{amsthm}
\usepackage{mathtools}
\usepackage[poly,arrow,curve,matrix]{xy}
\usepackage[a4paper]{geometry}
\usepackage{verbatim}
\usepackage{stmaryrd}
\usepackage{fullpage}
\usepackage{amscd}

\usepackage[backend=biber, style=numeric, maxnames=10]{biblatex}  
\usepackage{geometry}

\DeclareFontEncoding{OT2}{}{} 

\usepackage[colorlinks]{hyperref}
\usepackage{enumerate}
\usepackage[center]{caption}

\usepackage{tikz-cd}

\def\restrict#1{\raise-.5ex\hbox{\ensuremath|}_{#1}}

\def\XXint#1#2#3{{\setbox0=\hbox{$#1{#2#3}{\int}$ }
\vcenter{\hbox{$#2#3$ }}\kern-.5775\wd0}}

\newtheorem{theorem}{Theorem}[section]
\newtheorem{lemma}[theorem]{Lemma}
\newtheorem{proposition}[theorem]{Proposition}
\newtheorem{corollary}[theorem]{Corollary}
\newtheorem{remark}[theorem]{Remark}
\newtheorem{definition}[theorem]{Definition}

\newtheorem{conjecture}[theorem]{Conjecture}

\newcommand{\llangle}{\langle\!\langle}

\renewenvironment{proof}[1][Proof.]{\begin{trivlist}
\item[\hskip \labelsep {\itshape #1}]}{\end{trivlist}}

\usepackage[foot]{amsaddr}

\title{First-order definability of Campana Points and Darmon Points in algebraic function fields in one variable over number fields}
\author{Juan Pablo De Rasis}

\begin{document}

\maketitle

\begin{abstract}We give first-order definitions of Campana and Darmon points in algebraic function fields in one variable over number fields. These sets are geometric generalizations of $n$-full integers (integers whose nonzero valuations are at least $n$) and perfect $n$th powers, respectively, to more general algebraic varieties. For this we exploit the theory of quadratic Pfister forms, which were used in \cite{becher2025uniformexistentialdefinitionsvaluations} and \cite{daans2024universallydefiningsubringsfunction} to extend to the case of algebraic function fields in one variable the methods used by Koenigsmann in \cite{MR3432581}. These methods had already been generalized to arbitrary global fields by Park in \cite{park} and Eisenträger \& Morrison in \cite{MR3882159}, and the author had already exploited these methods to find first-order definitions of Campana points (in \cite{derasis2025firstorderdefinabilityaffinecampana}) and Darmon points (in \cite{derasis2024firstorderdefinabilitydarmonpoints}, with Handley) in the context of number fields. With the newly expanded version of these methods, we now transfer those results to the new context of algebraic function fields in one variable over number fields.
\end{abstract}

\section{Introduction}

If $R$ is an integral domain, $n\in\mathbb{Z}_{\geq 1}$ and $A\subseteq R^n$, we say that $A$ is \emph{diophantine} (or \emph{existentially definable}) over $R$ if there exists $m\in\mathbb{Z}_{\geq 0}$ and a polynomial $P\in R\left[x_1,\cdots,x_m,y_1,\cdots,y_n\right]$ such that $A=\left\{\vec{y}\in R^n:\exists x_1,\cdots,x_m\in R\left(P\left(x_1,\cdots,x_m,\vec{y}\right)=0\right)\right\}$. The understanding of diophantine sets is intimately related to the question about the decidability of the positive existential theory of the ring in question. Hilbert's Tenth Problem originally asked this question when $R=\mathbb{Z}$, which was negatively answered by Matiyasevich in \cite{MR0258744}, building its argument from Davis, Putnam, and Robinson's research on exponential diophantine equations in \cite{MR0133227}. Even though for various other rings an answer has also been attained (see for example \cite{MR0360513}, \cite{MR4126887}, \cite{MR4633727}, \cite{koymans2025hilbertstenthproblemadditive}), the decidability of the positive existential theory of $\mathbb{Q}$ and, more generally, any global field, remains unanswered. Given an integral domain $R$, the question about the decidability of the positive existential theory of $R$ is commonly known as \emph{Hilbert's Tenth Problem over $R$}.

If $R'$ is a subring of $R$, the diophantineness of $R'$ over $R$ allows to transfer a negative answer to Hilbert's Tenth Problem over $R'$ to $R$ (see for example \cite[Proposition 2.1]{garciafritz2023effectivity}), which explains the deep interest in the study of diophantine sets. In an attempt to contribute towards an existential definition of $\mathbb{Z}$ in $\mathbb{Q}$, Koenigsmann showed in \cite{MR3432581} that $\mathbb{Z}$ is \emph{universally} defined in $\mathbb{Q}$ (i.e. it is the complement of a diophantine set), by developing methods and results that were later generalized to prove analogous statements on more arbitrary global fields (see for example \cite{park}, \cite{MR3882159}, \cite{MR3778193}, \cite{morrisontravisnonnorms}, \cite{MR4378716}). Since the question about the undecidability of the positive existential theory of $\mathbb{Q}$ (or, more generally, a global field $K$) seems currently out of reach, research has tried an ``intermediate step'' approach, by studying intermediate subrings (some instances being \cite{MR1992832}, \cite{MR2549950}, \cite{MR2820576}, \cite{MR2915472}, \cite{MR3695862}). Motivated by this approach, the author offered a $\forall\exists$-definition for every member of a sequence of sets that interpolates between a number field and its set of $S$-integers. These sets are known as the \emph{Campana points} (see \cite{derasis2025firstorderdefinabilityaffinecampana}). Together with Handley, the author was able to extend his methods to give a similar approach for Darmon points in \cite{derasis2024firstorderdefinabilitydarmonpoints}.

Recently Becher, Daans, and Dittmann were able to further generalize Koenigsmann's methods to the case of algebraic function fields in one variable (see \cite{becher2025uniformexistentialdefinitionsvaluations} and \cite{daans2024universallydefiningsubringsfunction}). This motivated the author to transfer his previous results to this new scenario, and in this paper we prove:

\begin{theorem}[\ref{mainthmcampana} and \ref{mainthmdarmon}]\label{main}In algebraic function fields in one variable over number fields, Campana points are uniformly $\forall\exists$-definable and Darmon points are uniformly $\forall\exists\forall$-definable.
\end{theorem}

In the upcoming section we will see the specific definitions for Campana and Darmon points. It will be important to note from the beginning that these sets are given in terms of a finite set of $\mathbb{Z}$-valuations that are trivial on the base field. To be more specific, given an algebraic function field $F$ in one variable over a number field $K$, given $n\in\mathbb{Z}_{\geq 1}$ and given a finite set $S$ of $\mathbb{Z}$-valuations on $F$ that are trivial on $K$, we will define the sets $C_{F/K,S,n}$ (the set of $n$-Campana points of $F/K$ with respect to $S$) and $D_{F/K,S,n}$ (the set of $n$-Darmon points of $F/K$ with respect to $S$). As we will see, the first-order descriptions we will give will be uniform with respect to all possible choices of $S$ (cf. Remark \ref{campanauniform} and Remark \ref{darmonuniform}).

The above remark is important because, by \cite[Corollary 5.9]{daans2024universallydefiningsubringsfunction}, if $K$ is a number field and $F$ is an algebraic function field in one variable over $K$, every computably enumerable relation on $F$ is $\exists\forall$-definable. In particular, since every decidable subset is both computably enumerable and has a computably enumerable complement, then all decidable subsets are both $\forall\exists$-definable and $\exists\forall$-definable. Since Campana points and Darmon points are decidable, we can get our first-order characterization \emph{a priori}. The difference here is that our explicit approach towards a specific construction of their defining formulas will reveal to be \emph{uniform} with respect to all finite subsets $S$ of $\mathbb{Z}$-valuations trivial on the base field in terms of which these sets are defined. More specifically, the sequence of Proposition \ref{paramplaces}, Proposition \ref{paramplaces2}, and Proposition \ref{paramplaces3} will allow for a parametrization of all finite subsets of such places in a way that will allow for a uniform control on which of those sets $S$ is the one involved in the defining formula.

\subsection{Understanding Campana and Darmon points}

In the ring $\mathbb{Z}$, for each $n\in\mathbb{Z}_{\geq 1}$ we have the $n$-full integers; that is, integers $m$ such that $\nu_p\left(m\right)^2\geq n\nu_p\left(m\right)$ for all primes $p$. We also have the perfect $n$th powers. To generalize these notions to more arbitrary varieties, fix a number field $K$ and a smooth proper variety $X$ over $K$. Let $D=\displaystyle{\sum_{\alpha\in\mathcal{A}}\varepsilon_\alpha D_\alpha}$ be an effective Weil $\mathbb{Q}$-divisor, where $\mathcal{A}$ is a finite set and, for each $\alpha\in \mathcal{A}$, $D_\alpha$ is a prime divisor of $X$ and there exists $n_\alpha\in\mathbb{Z}_{\geq 1}\cup\left\{+\infty\right\}$ with $\varepsilon_\alpha=1-\frac{1}{n_\alpha}$. Assume also that $\displaystyle{\sum_{\alpha\in\mathcal{A}}D_\alpha}$ has strict normal crossings on $X$.

Given a finite subset $S$ of places in $K$ containing all the archimedean places, we take a model $\left(\mathcal{X},\mathcal{D}\right)$ of $\left(X,D\right)$ over $\mathcal{O}_{K,S}$ (the ring of $S$-integers of $K$; that is, the set of elements of $K$ which are integral outside $S$); i.e., a flat proper scheme $\mathcal{X}$ over $\mathcal{O}_{K,S}$ with $\mathcal{X}_{\left(0\right)}\cong X$ (where $\mathcal{X}_{\left(0\right)}$ is the generic fiber of $\mathcal{X}$ over $\mathcal{O}_{K,S}$), and $\mathcal{D}\coloneqq\displaystyle{\sum_{\alpha\in\mathcal{A}}\varepsilon_\alpha \mathcal{D}_\alpha}$. Here, for each $\alpha\in\mathcal{A}$, $\mathcal{D}_\alpha$ is the Zariski closure of $D_\alpha$ in $\mathcal{X}\hookleftarrow \mathcal{X}_{\left(0\right)}\cong X\supseteq D_\alpha$.

Given $P\in X\left(K\right)\coloneqq \operatorname{Hom}_{\operatorname{Sch}}\left(\operatorname{Spec}\left(K\right),X\right)$, the valuative criterion of properness (applied locally outside $S$ and then gluing all those local extensions) gives a unique $\mathcal{P}\in \mathcal{X}\left(\mathcal{O}_{K,S}\right)$ making the diagram\[
\begin{tikzcd}
\operatorname{Spec}\left(K\right) \arrow[d] \arrow[r, "P"]                                                               & X\cong \mathcal{X}_{\left(0\right)} \arrow[r, hook] & \mathcal{X} \arrow[d]                               \\
{\operatorname{Spec}\left(\mathcal{O}_{K,S}\right)} \arrow[rr, "\operatorname{id}"'] \arrow[rru, "\mathcal{P}"', dashed] &                                                     & {\operatorname{Spec}\left(\mathcal{O}_{K,S}\right)} \\
{\operatorname{Spec}\left(\mathcal{O}_{K,v}\right)} \arrow[u] \arrow[rruu, "\mathcal{P}_v", rounded corners, to path={ -- ([xshift=6cm]\tikztostart.east) |- (\tikztotarget)}]                                &                         ^{\mathcal{P}_v}                            &                                                    
\end{tikzcd}\]commute. For each place $v$ of $K$ outside $S$ take the $\mathcal{P}_v\in \mathcal{X}\left(\mathcal{O}_{K,v}\right)$ induced by the inclusion $\mathcal{O}_{K,S}\subseteq \mathcal{O}_{K,v}$ and define\[n_v\left(\mathcal{D}_\alpha,P\right)\coloneqq \begin{cases}+\infty,&P\in D_\alpha,\\\text{colength of the ideal of $\mathcal{O}_{K,v}$ corresponding to ${\mathcal{D}_\alpha}\times_\mathcal{X}\operatorname{Spec}\left(\mathcal{O}_{K,v}\right)$},&\mathcal{P}_v\not\subseteq \mathcal{D}_\alpha.\end{cases}\]We say that $P$ is:

\begin{itemize}

\item \textbf{Campana}, if for all $\alpha\in\mathcal{A}$ with $\varepsilon_\alpha=1$ we have $n_v\left(\mathcal{D}_\alpha,P\right)=0$, and for all $\alpha\in\mathcal{A}$ with $\varepsilon_\alpha<1$ and $n_v\left(\mathcal{D}_\alpha,P\right)>0$ we have $n_v\left(\mathcal{D}_\alpha,P\right)\geq n_\alpha$.

\item \textbf{Darmon}, if for all $\alpha\in\mathcal{A}$ with $\varepsilon_\alpha=1$ we have $n_v\left(\mathcal{D}_\alpha,P\right)=0$, and for all $\alpha\in\mathcal{A}$ with $\varepsilon_\alpha<1$ we have $n_\alpha\mid n_v\left(\mathcal{D}_\alpha,P\right)$.

\end{itemize}

For example, as shown in \cite[§3.1]{derasis2024firstorderdefinabilitydarmonpoints}, when $X=\mathbb{P}^1_\mathbb{Q}$, $\mathcal{X}=\mathbb{P}^1_\mathbb{Z}$, $D\coloneqq \left(1-\frac{1}{n}\right)\left\{x_0=0\right\}+\left\{x_1=0\right\}$ for some $n\in\mathbb{Z}_{\geq 1}$, and $S=\left\{\infty\right\}$, then the corresponding Darmon points in the affine line are the integral perfect $n$th powers up to sign. If instead $D\coloneqq \left(1-\frac{1}{n}\right)\left\{x_0=0\right\}+\left(1-\frac{1}{n}\right)\left\{x_1=0\right\}$ then we get \emph{rational} $n$th powers up to sign.

If $K$ is any number field, $S$ is a finite set of places of $K$ containing the archimedean ones, $X=\mathbb{P}^1_K$, $\mathcal{X}=\mathbb{P}^1_{\mathcal{O}_{K,S}}$, $n\in\mathbb{Z}_{\geq 1}$ and $D=\left(1-\frac{1}{n}\right)\left\{x_1=0\right\}$, then calling $C_{K,S,n}$ and $D_{K,S,n}$ the corresponding Campana and Darmon points, respectively, as shown in \cite[§2.2]{derasis2025firstorderdefinabilityaffinecampana} and \cite[§3.2]{derasis2024firstorderdefinabilitydarmonpoints} we get\begin{align*}C_{K,S,n}&=\left\{0\right\}\cup\left\{r\in K^\times:v\left(r\right)\in \mathbb{Z}_{\geq 0}\cup \mathbb{Z}_{\leq -n}\text{ for all places $v$ of $K$ outside $S$}\right\},\\D_{K,S,n}&=\left\{0\right\}\cup\left\{r\in K^\times:v\left(r\right)\in \mathbb{Z}_{\geq 0}\cup n\mathbb{Z}\text{ for all places $v$ of $K$ outside $S$}\right\}.\end{align*}These are the sets that were explored in \cite{derasis2025firstorderdefinabilityaffinecampana} and \cite{derasis2024firstorderdefinabilitydarmonpoints} and will be generalized in our new context.

\subsubsection{Other contexts in which these sets have been studied}

Even though these sets were mainly motivated by their interest in proving Manin-type conjectures and also in counting problems with height restrictions (see for instance \cite{balestrieri2023campana}, \cite{mitankin2023semiintegral}, \cite{MR4307130}, \cite{MR4405664}), they have also been studied in the context of definability and decidability. For example, in \cite{garciafritz2022nondiophantinesetsringsfunctions} Garcia-Fritz, Pasten, and Pheidas showed that Kollár's conjecture is incompatible with the diophantineness of Campana points in $\mathbb{C}\left(z\right)$, except for the trivial cases. We state the conjecture below. For each $n\in\mathbb{Z}_{\geq 1}$ we let $\mathbb{C}\left[z\right]_n$ be the $\left(n+1\right)$-dimensional $\mathbb{C}$-vector subspace of $\mathbb{C}\left[z\right]$ spanned by $\left\{1,x,\cdots,x^n\right\}$. We endow this set with the Zariski topology of $\mathbb{C}^{n+1}$.

\begin{conjecture}[Kollár]\label{kollar}If $D$ is a diophantine subset of $\mathbb{C}\left(z\right)$ containing a non-empty Zariski open subset of $\mathbb{C}\left[z\right]_n$ for infinitely many $n\in\mathbb{Z}_{\geq 1}$, then $\mathbb{C}\left(z\right)\setminus D$ is finite.
\end{conjecture}

Given a finite subset $S$ of $\mathbb{P}^1_\mathbb{C}\left(\mathbb{C}\right)=\mathbb{C}\cup\left\{\widehat{\infty}\right\}$ and $n\in\mathbb{Z}_{\geq 1}\cup\left\{\infty\right\}$ define\begin{align*}C_{S,n}&\coloneqq\left\{0\right\}\cup\left\{f\in\mathbb{C}\left(z\right)^\times:\forall \alpha\in\mathbb{P}^1_\mathbb{C}\left(\mathbb{C}\right)\setminus S\,\left(\nu_\alpha\left(f\right)\in\mathbb{Z}_{\geq 0}\cup\mathbb{Z}_{\leq -n}\right)\right\},\\D_{S,n}&\coloneqq\left\{0\right\}\cup\left\{f\in\mathbb{C}\left(z\right)^\times:\forall \alpha\in\mathbb{P}^1_\mathbb{C}\left(\mathbb{C}\right)\setminus S\,\left(\nu_\alpha\left(f\right)\in\mathbb{Z}_{\geq 0}\cup n\mathbb{Z}\right)\right\},\end{align*}where we adopt the convention $\infty\mathbb{Z}=\emptyset$.

It is clear that $C_{S,1}=D_{S,1}=\mathbb{C}\left(z\right)$ is diophantine in $\mathbb{C}\left(z\right)$. Moreover, $C_{\emptyset,\infty}=D_{\emptyset,\infty}=\mathbb{C}$ is also diophantine in $\mathbb{C}\left(z\right)$ because, in general, if $L$ is a large field and $F$ is an algebraic function field in one variable over $L$, then $L$ is diophantine over $F$, by \cite[Theorem 2]{377961c2-c185-3e97-9f96-fe7ef99cbae4}. We recall that a field $L$ is \emph{large} if and only if every smooth curve over $L$ has either infinitely many or none $L$-rational points.

In \cite[Proposition 1.13]{garciafritz2022nondiophantinesetsringsfunctions} Garcia-Fritz, Pasten, and Pheidas showed that, in the case of Campana points, these are the only instances of diophantineness that can occur if Conjecture \ref{kollar} is true. For the sake of completeness, let us now show how this argument can be adapted to get an analogous proposition for Darmon points.

\begin{proposition}Assume Conjecture \ref{kollar} is true and let $S$ be a finite subset of $\mathbb{C}\cup\left\{\widehat{\infty}\right\}$ and $n\in\mathbb{Z}_{\geq 1}\cup\left\{\infty\right\}$ be such that $D_{S,n}$ is diophantine over $\mathbb{C}\left(z\right)$. Then the following hold:

\begin{enumerate}[(i)]

\item If $n=\infty$ then $S=\emptyset$.

\item If $n\neq\infty$ then $n=1$.

\end{enumerate}

\end{proposition}

\begin{proof}Assume first $n=\infty$. If $S\neq\emptyset$, then up to a linear change of variables in $\mathbb{C}\left(z\right)$ (which does not alter diophantineness) we may assume $\widehat{\infty}\in S$. This implies that $D_{S,n}=D_{S,\infty}$ is the set of polynomials having positive valuation for all $\alpha\in\mathbb{C}\setminus S$ (recall our convention $\infty\mathbb{Z}=\emptyset$); in other words, $D_{S,n}$ is the set of $S$-integers of $\mathbb{C}\left(z\right)$, and it contains $\mathbb{C}\left[z\right]$. In particular it contains $\mathbb{C}\left[z\right]_{\ell}$ (which is itself an open Zariski subset of $\mathbb{C}\left[z\right]_{\ell}$) for all $\ell\in\mathbb{Z}_{\geq 0}$. This contradicts Conjecture \ref{kollar} because $\frac{1}{z-\lambda}\not\in D_{S,\infty}$ for all $\lambda\in\mathbb{C}\setminus S$.

Now assume $n\neq\infty$. Given any $\alpha\in\mathbb{Z}_{\geq 2}$ divisible by $n$, we claim that\[\left\{f\in\mathbb{C}\left[z\right]\setminus\left\{0\right\}:\deg\left(f\right)=\alpha\wedge f\left(0\right)\neq 0\right\}\subseteq D_{S,n}.\]Observe that the former set, under the identification $\mathbb{C}\left[z\right]_\alpha=\mathbb{C}^{\alpha+1}$, corresponds to the subset $\mathbb{C}^\times\times \mathbb{C}^{\alpha-1}\times\mathbb{C}^\times$, which is a Zariski open subset of $\mathbb{C}^{\alpha+1}$. This will imply, by Conjecture \ref{kollar}, that $D_{S,n}$ is cofinite in $\mathbb{C}\left(z\right)$; but if $n>1$ we have $\frac{1}{z-\lambda}\not\in D_{S,n}$ for all $\lambda\in\mathbb{C}\setminus S$, a contradiction.

Fix a nonzero $f\in\mathbb{C}\left[z\right]$ of degree $\alpha$. We must have $\nu_{\lambda}\left(f\right)\geq 0$ for all $\lambda\in \mathbb{C}$, and moreover $\nu_{\widehat{\infty}}\left(f\right)=-\deg\left(f\right)=-\alpha\in n\mathbb{Z}$, so $f\in D_{S,n}$, as desired. $\blacksquare$

\end{proof}

\subsection{Roadmap}

Sections \ref{valuations}, \ref{quadforms}, and \ref{dim2k-cd2} will be used to discuss the foundations of the results of this paper. We start with Section \ref{valuations} by recalling the notions of valuations and extensions, and we continue with Section \ref{quadforms} to review the basics of quadratic forms and Witt equivalence, an important part of which will be the study of quadratic Pfister forms. We will conclude the preliminary sections by discussing $p$-th cohomological dimension (for a given prime $p>0$) in Section \ref{valuations}. We combine these sections to state and understand the Reciprocity Theorem of the Witt group (Theorem \ref{recthm}) in Section \ref{reciprocitypfister}.

Once these are established, in Section \ref{methodology} we will give an overview of the method we will use to attain our desired first-order definitions, and will prove the required lemmas that ensure its efficacy. More concretely, we will plan to reduce our problem to the case of \emph{regular} algebraic function fields in one variable over \emph{non-real} number fields, for which we will need a reduction step (given by Theorem \ref{reductionstep}) and then some sanity checks in terms of first-order interpretation (Section \ref{coopolynomials}) and non-ramification (Section \ref{nonramification}). Finally, we begin Section \ref{setupandproofs} by setting up the notations and preliminary first-order descriptions that we will use in Section \ref{provingcampana} and Section \ref{provingdarmon} to demonstrate Theorem \ref{main}.

\section{Acknowledgments}

I thank my supervisor Dr. Jennifer Park for her guidance, suggestions, proof-readings, and comments, and for being the principal investigator of the NSF grant DMS-2152182 that funded this research. I also thank the \emph{Hausdorff Research Institute for Mathematics} from Universität Bonn, funded by the DFG under Germany's Excellence Strategy (EXC-2047/1 - 390685813), for having hosted the $2025$ trimester program \emph{Definability, decidability, and computability}, organized by Valentina Harizanov, Philipp Hieronymi, Jennifer Park, Florian Pop, and Alexandra Shlapentokh, in which I could solve all my difficulties regarding the content of this paper. I also thank Dr. Philip Dittmann for the extremely useful conversations we held during that program (which were the main reason why I could finally resolve all my issues and better understand his and his coauthors papers, which offer the foundations on which the current paper relies) and for his useful comments after having read a first draft. Finally, I thank my colleague Hunter Handley for his careful proof-reading and corrections.

\section{$\mathbb{Z}$-valuations on algebraic function fields in one variable}\label{valuations}

Let us recall the following definitions from field theory:

\begin{definition}If $K$ is a field and $L$ is a field extension of $K$, we say that the extension $L/K$ is \emph{regular} if and only if it is separable and the only elements of $L$ algebraic over $K$ are the elements of $K$ itself.
\end{definition}

\begin{definition}Given a field $K$, an \emph{algebraic function field in one variable over $K$} is a field $F$ such that $F/K$ is a finitely generated field extension of transcendence degree $1$.
\end{definition}

Now let us fix some notation for later use. Given a field $K$ and a discrete valuation $v$ on $K$, we denote\begin{align*}\mathcal{O}_v&\coloneqq \left\{x\in K:v\left(x\right)\geq 0\right\}\text{ (the \emph{valuation ring} associated to $v$)},\\\mathfrak{m}_v&\coloneqq \left\{x\in K:v\left(x\right)>0\right\}\text{ (the \emph{maximal ideal} of $v$)},\\Kv&\coloneqq \frac{\mathcal{O}_v}{\mathfrak{m}_v}\text{ (the \emph{residue field} of $v$)}.\end{align*}Observe that, under this notation, $\mathcal{O}_v$ is a local ring with maximal ideal $\mathfrak{m}_v$, thus we get $\mathcal{O}_v^\times=\mathcal{O}_v\setminus\mathfrak{m}_v=\left\{x\in K:v\left(x\right)=0\right\}$. Finally, we denote $K_v$ (not to be confused with the residue field $Kv$) as the \emph{henselization} of $K$ with respect to $v$; that is, the separable closure of $K$ in its completion with respect to $v$.

Recall that a $\mathbb{Z}$-valuation $v$ on a field $K$ is a discrete valuation such that for every $n\in\mathbb{Z}$ there exists $x\in K$ such that $n=v\left(x\right)$. If $L$ is a finite field extension of $K$ and $w$ is a $\mathbb{Z}$-valuation on $L$, we say that \emph{$w$ lies over $v$} (or that \emph{$v$ lies under $w$}) if and only if $\mathcal{O}_w\cap K=\mathcal{O}_v$. In this case, there exists $e\left(w\mid v\right)\in\mathbb{Z}_{\geq 1}$ such that $w\left(x\right)=e\left(w\mid v\right)v\left(x\right)$ for every $x\in K$. The number $e\left(w\mid v\right)$ is called the \emph{ramification index of $w$ over $v$}. We say that $v$ is \emph{unramified in $L$} if $e\left(w\mid v\right)=1$ for every $\mathbb{Z}$-valuation on $L$ lying over $K$. It is well-known that any $\mathbb{Z}$-valuation on $K$ lies under only finitely many (and at least one) $\mathbb{Z}$-valuations on $L$.

We refer the reader to \cite{Neukirch1999} and \cite{Fried2023} for more properties of valuations. In our case, we will be considering $\mathbb{Z}$-valuations on algebraic function fields in one variable.

Let $K$ be a field, and let $F$ be an algebraic function field in one variable over $K$. We let $\mathcal{V}\left(F/K\right)$ be the set of $\mathbb{Z}$-valuations on $F$ which are trivial on $K$, and denote $\mathcal{O}_{F/K}\coloneqq \displaystyle{\bigcap_{v\in\mathcal{V}\left(F/K\right)}\mathcal{O}_v}$. If $L$ is a finite extension of $K$ contained in the algebraic closure of $F$, then $\left[FL:F\right]\leq \left[L:K\right]$ and therefore\begin{multline*}\operatorname{tr}\deg\left(FL/L\right)=\operatorname{tr}\deg\left(FL/K\right)-\operatorname{tr}\deg\left(L/K\right)=\operatorname{tr}\deg\left(FL/K\right)\\=\operatorname{tr}\deg\left(FL/F\right)+\operatorname{tr}\deg\left(F/K\right)=\operatorname{tr}\deg\left(F/K\right)=1,\end{multline*}thus $FL$ is an algebraic function field in one variable over $L$, and every $v\in\mathcal{V}\left(F/K\right)$ lies under only finitely many (and at least one) elements of $\mathcal{V}\left(FL/L\right)$.

We now define the objects we desire to describe with first-order language:

\begin{definition}Let $K$ be a field and let $F$ be an algebraic function field in one variable over $K$. For every $n\in\mathbb{Z}_{\geq 1}$ and every finite subset $S$ of $\mathcal{V}\left(F/K\right)$ we define\begin{align*}C_{F/K,S,n}&\coloneqq\left\{0\right\}\cup\left\{x\in F^\times:v\left(x\right)\in\mathbb{Z}_{\geq 0}\cup\mathbb{Z}_{\leq -n}\text{ for all $v\in\mathcal{V}\left(F/K\right)\setminus S$}\right\},\\D_{F/K,S,n}&\coloneqq\left\{0\right\}\cup\left\{x\in F^\times:v\left(x\right)\in\mathbb{Z}_{\geq 0}\cup n\mathbb{Z}\text{ for all $v\in\mathcal{V}\left(F/K\right)\setminus S$}\right\}.\end{align*}The set $C_{F/K,S,n}$ is called \emph{the set of $n$-Campana points of $F/K$ with respect to $S$}, while the set $D_{F/K,S,n}$ is called \emph{the set of $n$-Darmon points of $F/K$ with respect to $S$}.
\end{definition}

\section{Quadratic Forms and Witt Equivalence}\label{quadforms}

Given a field $K$, let us recall that a \emph{quadratic form over $K$} is a function $q:V\to K$ (where $V$ is a finite-dimensional vector space over $K$) such that $q\left(\lambda v\right)=\lambda^2 v$ for all $\lambda\in K$ and $v\in V$, and such that the map $V\times V\to K$ given by $\left(v,w\right)\mapsto q\left(v+w\right)-q\left(v\right)-q\left(w\right)$ for all $\left(v,w\right)\in V\times V$ is $K$-bilinear. We denote $\dim_K\left(q\right)\coloneqq \dim_K\left(V\right)$. We say that $q$ is \emph{isotropic} if $q\left(v\right)=0$ for some $v\in V\setminus\left\{0\right\}$, and \emph{anisotropic} otherwise. We say that $q$ is \emph{non-singular} if for every $v\in V\setminus\left\{0\right\}$ there exists $w\in V$ such that $q\left(v+w\right)\neq q\left(v\right)+q\left(w\right)$.

Given two finite-dimensional $K$-vector spaces $V_1$ and $V_2$, and quadratic forms $q_1:V_1\to K$ and $q_2:V_2\to K$ over $K$, we define the \emph{orthogonal product of $q_1$ and $q_2$} as the quadratic form $q_1\perp q_2:V_1\oplus V_2\to K$ given by $\left(v_1,v_2\right)\mapsto q_1\left(v_1\right)+q_2\left(v_2\right)$ for every $\left(v_1,v_2\right)\in V_1\times V_2$. Also, we say that $q_1$ and $q_2$ are \emph{isometric} if and only if there exists a $K$-isomorphism $\varphi:V_1\to V_2$ such that $q_2\circ \varphi=q_1$.

For a quadratic form $q:V\to K$ on a finite-dimensional vector space $V$ over a field $K$ we define\[\operatorname{rad}\left(q\right)\coloneqq \left\{v\in V:q\left(v\right)=0\wedge\forall w\in V\left(q\left(v+w\right)=q\left(v\right)+q\left(w\right)\right)\right\},\]which is a $K$-vector subspace of $V$, and it is the null space if $q$ is non-singular. We also define the quadratic form $\mathbb{H}\left(V\right):V\oplus V^\ast\to K$ as $\mathbb{H}\left(V\right)\left(v,\sigma\right)\coloneqq\sigma\left(v\right)$ for every $\left(v,\sigma\right)\in V\times V^\ast$. A quadratic form isometric to $\mathbb{H}\left(W\right)$ for some finite-dimensional vector space $W$ of $K$ is called \emph{hyperbolic}. Observe that hyperbolic quadratic forms are even-dimensional and closed under orthogonal product.

\begin{lemma}\label{wittinversion}Let $K$ be a field and let $q$ be a non-singular quadratic form over $K$. Then there exists another non-singular quadratic form $q'$ over $K$ such that $q\perp q'$ is hyperbolic and $\dim_K\left(q\right)\equiv\dim_K\left(q'\right)\pmod{2}$.
\end{lemma}

\begin{proof}Assume fist $\operatorname{Char}\left(K\right)\neq 2$. By the diagonalization theorem for quadratic forms on characteristic different from $2$ (\cite[Proposition 7.29]{Elman2008}) and the fact that $q$ is non-singular, $q$ is isometric to the orthogonal product of quadratic forms $K\to K$ given by $x\mapsto ax^2$ for every $x\in K$, where $a\in K^\times$. Calling $\left\langle a\right\rangle$ such a quadratic form for a given $a\in K^\times$, the first conclusion follows from the fact that $\left\langle a\right\rangle\perp \left\langle -a\right\rangle$ is hyperbolic (\cite[Example 7.26]{Elman2008}). The second conclusion is a direct consequence of hyperbolic forms always being even-dimensional.

The proof when $\operatorname{Char}\left(K\right)=2$ is essentially the same; with the additional task of dealing with the more complicated version of the orthogonal factorization of quadratic forms on fields of characteristic $2$ (\cite[Proposition 7.31]{Elman2008}), since we have to also consider orthogonal factors $K\times K\to K$ of the form $\left(x,y\right)\mapsto ax^2+xy+by^2$ with $a,b\in K$. Calling $\left[a,b\right]$ such a quadratic form for given $a,b\in K$, \cite[Example 7.23]{Elman2008} gives $\left[a,b\right]\perp \left[a,-b\right]\cong \left[a,0\right]\perp \mathbb{H}\left(K\right)$, and \mbox{\cite[Example 7.5]{Elman2008}} gives $\left[a,0\right]\cong \mathbb{H}\left(K\right)$. $\blacksquare$
\end{proof}

\begin{theorem}[Witt Decomposition Theorem]\label{wittdecom}Let $V$ be a finite-dimensional vector space over a field $K$ and let $q:V\to K$ be a quadratic form over $K$. Then there exist $K$-vector subspaces $U$ and $W$ of $V$ such that $V=\operatorname{rad}\left(q\right)\oplus U\oplus W$ and\[q=q\restrict {\operatorname{rad}\left(q\right)}\perp q\restrict U\perp q\restrict W,\]where $q\restrict U$ is anisotropic and $q\restrict W$ is hyperbolic. Moreover, such forms $q\restrict U$ and $q\restrict {W}$ are unique up to isometry.
\end{theorem}

\begin{proof}See \cite[Theorem 8.5]{Elman2008}. $\blacksquare$
\end{proof}

Theorem \ref{wittdecom} implies that for every quadratic form $q$ over a field $K$ we have a well-defined isometry class $q_{\operatorname{an}}$ given by the isometry class of the anisotropic part of the Witt decomposition of $q$. If $q'$ is another quadratic form over $K$, we say that $q$ and $q'$ are \emph{Witt equivalent} if $\dim_K\left(\operatorname{rad}\left(q\right)\right)=\dim_K\left(\operatorname{rad}\left(q'\right)\right)$ and $q_{\operatorname{an}}=q'_{\operatorname{an}}$. Since this equivalence is respected by orthogonal products, we have a well-defined monoid. Lemma \ref{wittinversion} guarantees that the monoid operation has inverses, and that these are well-defined even if we restrict ourselves to even-dimensional quadratic forms. We then have the following well-defined group:

\begin{definition}Given a field $K$, we define the \emph{Witt Group} of $K$ as the set of Witt classes of the non-singular even-dimensional quadratic forms over $K$, with orthogonal product as its group operation. We denote $WG\left(K\right)$ this group.
\end{definition}

Given any quadratic form $q$ we will denote $\left[q\right]$ its Witt class.

\subsection{Quadratic Pfister forms}\label{pfisterintro}

The quadratic forms that will play an essential role in this paper are the \emph{Quadratic Pfister forms}, which we now define.

\begin{definition}Given a field $K$ and $b\in K$ such that $1+4b\neq 0$, we define the $2$-dimensional quadratic form $\llangle b]]_K:K\times K\to K$ as $\left(x,y\right)\mapsto x^2-xy-by^2$ for every $\left(x,y\right)\in K\times K$. Such a form is called \emph{a $1$-fold quadratic Pfister form over $K$}. Recursively, if for some $n\in\mathbb{Z}_{\geq 1}$ and $a_1,\cdots,a_{n-1}\in K^\times$ and $b\in K$ with $1+4b\neq 0$ we have defined an $n$-fold quadratic Pfister form $q=\llangle a_1,\cdots,a_{n-1},b]]_K:V\to K$ over $K$ for some $K$-vector space $V$ of dimension $2^{n}$, for $a_n\in K^\times$ we define the quadratic form\begin{equation}\label{defpfister}\begin{matrix}\llangle a_1,\cdots,a_n,b]]_K&:&V\times V&\to&K,\\&&\left(v_1,v_2\right)&\mapsto&q\left(v_1\right)-a_nq\left(v_2\right).\end{matrix}\end{equation}We call this \emph{an $\left(n+1\right)$-fold quadratic Pfister form over $K$.}
\end{definition}

The first important fact about quadratic Pfister forms is that their Witt decomposition has only two possibilities.

\begin{proposition}\label{wittdecompfister}Given a field $K$ and a quadratic Pfister form $q$ over $K$, either $q$ is anisotropic or $q$ is hyperbolic.
\end{proposition}

\begin{proof}See \cite[Corollary 9.10]{Elman2008}. $\blacksquare$
\end{proof}

A direct consequence of Proposition \ref{wittdecompfister} is that a quadratic Pfister form is anisotropic if and only if it has nontrivial Witt class. We now split the Witt classes of these forms into levels.

\begin{definition}Given a field $K$, we define $WG^1\left(K\right)\coloneqq WG\left(K\right)$ and, for $n\in\mathbb{Z}_{\geq 2}$, we let $WG^n\left(K\right)$ be the subgroup of $WG\left(K\right)$ generated by the Witt classes of all $n$-fold quadratic Pfister forms over $K$.
\end{definition}

\begin{remark}\label{witt-pfister-chain}If $K$ is a field and $n\in\mathbb{Z}_{\geq 1}$, by \eqref{defpfister} we have $WG^{n+1}\left(K\right)\subseteq WG^n\left(K\right)$. Since the Witt group is abelian, we can define the quotient group $WG^n\left(K\right)/WG^{n+1}\left(K\right)$.
\end{remark}

If $K$ is a field, $n\in\mathbb{Z}_{\geq 1}$, $a_1,\cdots,a_n\in K^\times$, and $b\in K$ is such that $1+4b\neq 0$, then \eqref{defpfister} gives\[\llangle a_1,\cdots,a_n,b]]_K\perp \llangle a_1,\cdots,a_n,b]]_K=\llangle a_1,\cdots,a_n,-1,b]]_K,\]so we immediately get:

\begin{lemma}\label{2-torsion}Let $K$ be a field and $n\in\mathbb{Z}_{\geq 1}$. The group $WG^n\left(K\right)/WG^{n+1}\left(K\right)$ is $2$-torsion.
\end{lemma}

Witt equivalence is stable under field extensions. This is, given a field $K$ and two quadratic forms $q:V\to K$ and $q':V'\to K$, where $V$ and $V'$ are finite-dimensional $K$-vector spaces such that $\left[q\right]=\left[q'\right]\in WG\left(K\right)$, we then have $\left[q_L\right]=\left[q'_L\right]\in WG\left(L\right)$, where for any quadratic form $\varphi:W\to K$ over $K$ where $W$ is a finite-dimensional $K$-vector space we define $\varphi_L:W\otimes_KL\to L$ to be the quadratic form over $L$ satisfying $\varphi_L\left(w\otimes \lambda\right)=\lambda^2\varphi\left(w\right)$ for every $\left(w,\lambda\right)\in W\times L$. Moreover, if $n\in\mathbb{Z}_{\geq 1}$ and $\left[\varphi\right]\in WG\left(K\right)$ equals the Witt class of some $n$-fold quadratic Pfister form over $K$, then $\left[\varphi_L\right]\in WG\left(L\right)$ equals the Witt class of some $n$-fold quadratic Pfister form over $L$. We then have the following well-defined group homomorphism:

\begin{proposition}For a field $K$ and a field extension $L$ over $K$, the map $WG\left(K\right)\to WG\left(L\right)$ given by $\left[q\right]\mapsto \left[q_L\right]$ for every quadratic form $q$ over $K$ is called \emph{the restriction homomorphism with respect to $L/K$}. This is a well-defined group homomorphism which for every $n\in\mathbb{Z}_{\geq 1}$ restricts and co-restricts to a group homomorphism $WG^n\left(K\right)\to WG^n\left(L\right)$.
\end{proposition}

A key concept regarding the quotient group defined in Remark \ref{witt-pfister-chain} is that of \emph{linkage} of quadratic Pfister forms, which says that Pfister forms themselves are sufficient to characterize the elements of the quotient.

\begin{definition}\label{linkage}Given a field $K$ and $n\in\mathbb{Z}_{\geq 1}$, we say that $n$-fold quadratic Pfister forms are \emph{linked over $K$} if and only if every element of $WG^n\left(K\right)/WG^{n+1}\left(K\right)$ is the class of an $n$-fold quadratic Pfister form over $K$.
\end{definition}

\subsubsection{Scharlau Transfer Maps}

Assume $L$ is a finite field extension of a field $K$, and that $s:L\to K$ is any nonzero $K$-linear map. For any quadratic form $q$ over $L$, $s\circ q$ is a quadratic form over $K$. Moreover, if $q$ is non-singular, then so is $s\circ q$ (\cite[Lemma 20.4(2)]{Elman2008}), and if $q$ is hyperbolic, then so is $s\circ q$ (\cite[Corollary 20.5(2)]{Elman2008}). Therefore we have a well-defined group homomorphism $s_\ast:WG\left(L\right)\to WG\left(K\right)$ given by $\left[q\right]\mapsto \left[s\circ q\right]$ for every quadratic form $q$ over $L$.

\begin{lemma}\label{independenceofchoiceofsatdim}Let $L$ be a finite field extension of a field $K$, and let $d\in\mathbb{Z}_{\geq 1}$ be such that $WG^{d+1}\left(L\right)$ is trivial. If $s,t\in L^\ast$ are nonzero, then $s_\ast\restrict{WG^d\left(L\right)}=t_\ast\restrict{WG^d\left(L\right)}$.
\end{lemma}

\begin{proof}This is the same proof as in \cite[Lemma 2.2]{daans2024universallydefiningsubringsfunction}. The $K$-linear map $\mu:L\to L^\ast$ satisfying $\mu\left(a\right)\left(x\right)=s\left(ax\right)$ for every $a,x\in L$ is a monomorphism. Indeed, if $a\in\operatorname{ker}\left(\mu\right)$ then $s\left(ax\right)=0$ for every $x\in L$. If $a\neq 0$ then $s$ would be the zero element of $L^\ast$, a contradiction. Thus $a=0$ and $\mu$ is a monomorphism. Since $L$ is finite-dimensional over $K$ then $\dim_K\left(L\right)=\dim_K\left(L^\ast\right)$, thus $\mu$ is an isomorphism. We conclude that $t=\mu\left(a\right)$ for some $a\in L$; that is, $t\left(x\right)=s\left(ax\right)$ for all $x\in L$. We have $a\in L^\times$ because $t$ is not the zero map.

Given $\alpha\in WG^d\left(L\right)$, we then have $t_{\ast}\left(\alpha\right)=s_\ast\left(a\alpha\right)$. Moreover, \eqref{defpfister} gives $\alpha-a\alpha\in WG^{d+1}\left(L\right)$, thus the hypothesis implies $\alpha=a\alpha\in WG\left(L\right)$. We conclude that $t_{\ast}\left(\alpha\right)=s_\ast\left(\alpha\right)$, as desired. $\blacksquare$

\end{proof}

\section{The $2$-dimension of a field}\label{dim2k-cd2}

For this section we require to review the concept of cohomological dimension with respect to a prime number. Let us recall that for any abelian group $A$ and a prime $p>0$ we define $A\left[p^\infty\right]\coloneqq \left\{a\in A:\exists n\in\mathbb{Z}_{\geq 0}\left(p^na=0\right)\right\}$, which is a subgroup of $A$.

\begin{definition}Given a profinite group $G$, for each prime $p>0$ we define the \emph{$p$-th cohomological dimension of $G$}, and denote it $\operatorname{cd}_p\left(G\right)$, as the smallest $n\in\mathbb{Z}_{\geq 1}$ such that $H^n\left(G,A\right)\left[p^\infty\right]$ is trivial for every torsion $G$-module $A$ satisfying that the action map $G\times A\to A$ is continuous with respect to the profinite topology of $G$ and the discrete topology of $A$.

If no such $n$ exists, we define $\operatorname{cd}_p\left(G\right)\coloneqq +\infty$.
\end{definition}

\begin{definition}Given a field $K$ with separable closure $K_s$, for every prime $p>0$ we define\[\operatorname{cd}_p\left(K\right)\coloneqq \operatorname{cd}_p\left(\operatorname{Gal}\left(K_s/K\right)\right).\]
\end{definition}

For the basic properties of the cohomological dimension we refer the reader to \cite[§1.4]{MR3729254}. In our case, we will be interested only in the following three properties:

\begin{proposition}\label{dim2andtrdeg}Let $K$ be a field, let $L$ be a finitely generated field extension of $K$, and let $p>0$ be a prime such that $p\neq \operatorname{Char}\left(K\right)$ and $\operatorname{cd}_p\left(K\right)\neq+\infty$. We then have\[\operatorname{cd}_p\left(L\right)=\operatorname{cd}_p\left(K\right)+\operatorname{tr}\operatorname{deg}\left(L/K\right).\]
\end{proposition}

\begin{proof}See \cite[§4.2, Proposition 11]{Serre1997}. $\blacksquare$
\end{proof}

\begin{proposition}\label{dim2andplace}Let $K$ be a field which is complete with respect to a discrete valuation $v$ and such that $Kv$ is a perfect field. If $p>0$ is a prime with $p\neq \operatorname{Char}\left(K\right)$ and $\operatorname{cd}_p\left(K\right)\neq+\infty$, then\[\operatorname{cd}_p\left(K\right)=\operatorname{cd}_p\left(Kv\right)+1.\]
\end{proposition}

\begin{proof}See \cite[§5, Theorem 6.5.15]{Neukirch2008}. $\blacksquare$
\end{proof}

\begin{proposition}\label{dim2k}If $K$ is a field and $\operatorname{Char}\left(K\right)\neq 2$, then for all $n\in\mathbb{Z}_{\geq 1}$ the following are equivalent:

\begin{enumerate}

\item $\operatorname{cd}_2\left(K\right)\leq n$.

\item For every field $L$ that is a separable finite extension of $K$, the group $WG^{n+1}\left(L\right)$ is trivial.

\end{enumerate}
\end{proposition}

\begin{proof}See \cite[Proposition 3.1]{daans2024universallydefiningsubringsfunction}. $\blacksquare$
\end{proof}

\section{Reciprocity theorem for quadratic Pfister forms}\label{reciprocitypfister}

In this section we put together Sections \ref{valuations}, \ref{quadforms}, and \ref{dim2k-cd2}. We start with a lemma that allows for the rewriting of the parameters of quadratic Pfister forms up to isometry.

\begin{proposition}\label{rewrparampfister}Let $K$ be a field, $n\in\mathbb{Z}_{\geq 1}$, $q$ be an $\left(n+1\right)$-fold quadratic Pfister form over $K$, and $S$ a finite set of $\mathbb{Z}$-valuations on $K$. Then\[q\cong \llangle a_1,\cdots,a_n,b]]_K\]for some $a_1,\cdots,a_n,b\in \displaystyle{\bigcap_{v\in S}\mathcal{O}_v}$ such that $a_n\neq 0$ and $a_1,\cdots,a_{n-1},1+4b\in\displaystyle{\bigcap_{v\in S}\mathcal{O}_v^\times}$.
\end{proposition}

\begin{proof}Follows directly from \cite[Lemma 2.1]{daans2024universallydefiningsubringsfunction}. $\blacksquare$
\end{proof}

Let $K$ be a field, $v$ a henselian $\mathbb{Z}$-valuation on $K$, $V$ a finite-dimensional $K$-vector space, and $q:V\to K$ an anisotropic quadratic form over $K$. We define the following $\mathcal{O}_v$-submodules of $V$:\begin{align*}V_{q,v}&\coloneqq \left\{x\in V:v\left(q\left(x\right)\right)\geq 0\right\},\\V_{q,v}^0&\coloneqq\left\{x\in V:v\left(q\left(x\right)\right)>0\right\}.\end{align*}If we regard $V_{q,v}/V_{q,v}^0$ as a finite-dimensional vector space over $Kv=\mathcal{O}_v/\mathfrak{m}_v$, then we have an anisotropic quadratic form $r_v\left(q\right):V_{q,v}/V_{q,v}^0\to Kv$ over $Kv$ given by $r_v\left(q\right)\left(\overline{x}\right)=\overline{q\left(x\right)}$ for every $x\in V_{q,v}$, where for every $v\in V_{q,v}$ we denote $\overline{v}$ its class modulo $V_{q,v}^0$, and for every $\lambda\in \mathcal{O}_v$ we denote $\overline{\lambda}$ its class modulo $\mathfrak{m}_v$.

\begin{proposition}\label{residuemaps}Let $K$ be a field, let $v$ be a $\mathbb{Z}$-valuation on $K$ with $\operatorname{Char}\left(Kv\right)\neq 2$, and let $n\in\mathbb{Z}_{\geq 1}$. Then there exists a unique surjective group homomorphism\[\partial_v^n:WG^{n+1}\left(K\right)/WG^{n+2}\left(K\right)\to WG^n\left(Kv\right)/WG^{n+1}\left(Kv\right)\]such that for every non-singular even-dimensional quadratic form $q$ over $K$ with $\left[q\right]\in WG^{n+1}\left(K\right)$ the following hold:

\begin{itemize}
\item If $q_{K_v}$ is anisotropic, then $\partial_v^n\left(\left[q\right]+WG^{n+2}\left(K\right)\right)=\left[r_v\left(q_{K_v}\right)\right]+WG^{n+1}\left(Kv\right)$.

\item If $q_{K_v}$ is hyperbolic, then $\left[q\right]+WG^{n+2}\left(K\right)\in \operatorname{ker}\left(\partial_v^n\right)$.
\end{itemize}

Moreover, this $\partial_v^n$ satisfies:

\begin{enumerate}[(i)]

\item If $q$ is an $\left(n+1\right)$-fold quadratic Pfister form over $K$, then\[\partial_v^n\left(\left[q\right]+WG^{n+2}\left(K\right)\right)=\left[\varphi\right]+WG^{n+1}\left(Kv\right)\]for some $n$-fold quadratic Pfister form $\varphi$ over $Kv$.

\item If $v$ is henselian and $WG^{n+1}\left(Kv\right)$ is trivial, then $\partial_v^n$ is an isomorphism.

\end{enumerate}
\end{proposition}

\begin{proof}This is a particular case of \cite[Proposition 2.3]{daans2024universallydefiningsubringsfunction}. $\blacksquare$
\end{proof}








Given a field $K$ with $\operatorname{Char}\left(K\right)\neq 2$ and $d\coloneqq\operatorname{cd}_2\left(K\right)\in\mathbb{Z}_{\geq 1}$, Proposition \ref{dim2k} gives that $WG^{d+1}\left(K\right)$ is trivial. If $L$ is a finite field extension of $K$, Lemma \ref{independenceofchoiceofsatdim} says that we have a well-defined group homomorphism $WG^d\left(L\right)\to WG\left(K\right)$ given by $\left[q\right]\mapsto \left[s\circ q\right]$ for every quadratic form $q$ over $L$ with $\left[q\right]\in WG^d\left(L\right)$, where $s:L\to K$ is any nonzero $K$-linear map.

\begin{definition}If $K$ is a field with $\operatorname{Char}\left(K\right)\neq 2$ and $d\coloneqq\operatorname{cd}_2\left(K\right)\in\mathbb{Z}_{\geq 1}$, and $L$ is a finite extension of $K$, we denote $s_{L/K}$ the (well-defined) group homomorphism $WG^d\left(L\right)\to WG\left(K\right)$ given by $\left[q\right]\mapsto \left[t\circ q\right]$ for every even-dimensional non-singular quadratic form $q$ over $L$ with $\left[q\right]\in WG^d\left(L\right)$, where $t:L\to K$ is any nonzero $K$-linear map.
\end{definition}

\begin{lemma}\label{surjectivityoftransfer}If $K$ is a field with $\operatorname{Char}\left(K\right)\neq 2$ and $d\coloneqq\operatorname{cd}_2\left(K\right)\in\mathbb{Z}_{\geq 1}$, and $L$ is a finite extension of $K$, then $\operatorname{Im}\left(s_{L/K}\right)=WG^d\left(K\right)$.
\end{lemma}

\begin{proof}This is a particular case of \cite[Lemma 3.8]{daans2024universallydefiningsubringsfunction}. $\blacksquare$
\end{proof}

Now let $K$ be a field with $\operatorname{Char}\left(K\right)\neq 2$ and $d\coloneqq\operatorname{cd}_2\left(K\right)\in\mathbb{Z}_{\geq 1}$, and let $F$ be an algebraic function field in one variable over $K$. Proposition \ref{dim2andtrdeg} gives $\operatorname{cd}_2\left(F\right)=d+1$, and \mbox{Proposition \ref{dim2andplace}} gives $\operatorname{cd}_2\left(Fv\right)=\left(d+1\right)-1=d$ for all $v\in\mathcal{V}\left(F/K\right)$. Fixing such a $v$, Proposition \ref{dim2k} says that the groups $WG^{d+2}\left(F\right)$ and $WG^{d+1}\left(Fv\right)$ are trivial, thus Lemma \ref{residuemaps} (which can be applied since $K$ is a subfield of $Fv$ and so $\operatorname{Char}\left(Fv\right)=\operatorname{Char}\left(K\right)\neq 2$) gives\[\partial_v^d:WG^{d+1}\left(F\right)\to WG^{d}\left(Fv\right).\]

With all this in mind, we can now state the important reciprocity theorem that lies at the heart of our arguments:

\begin{theorem}[Reciprocity Theorem]\label{recthm}Let $K$ be a field with $\operatorname{Char}\left(K\right)\neq 2$ and \mbox{$d\coloneqq\operatorname{cd}_2\left(K\right)\in\mathbb{Z}_{\geq 1}$,} and let $F$ be a regular algebraic function field in one variable over $K$. Then the sequence\begin{equation}\label{reciprocitysequence}
\begin{tikzcd}
WG^{d+1}\left(F\right) \arrow[r] & \displaystyle{\bigoplus_{v\in\mathcal{V}\left(F/K\right)}WG^d\left(Fv\right)} \arrow[r] & WG^d\left(K\right) \arrow[r] & 0
\end{tikzcd}\end{equation}where the first arrow is induced by the maps $\partial_v^d:WG^{d+1}\left(F\right)\to WG^{d}\left(Fv\right)$ for all $v\in\mathcal{V}\left(F/K\right)$ and the second arrow is induced by the maps $s_{Fv/K}$ for all $v\in\mathcal{V}\left(F/K\right)$, is well-defined and exact.
\end{theorem}

\begin{proof}See \cite[Theorem 3.2]{daans2024universallydefiningsubringsfunction}. $\blacksquare$
\end{proof}

Observe that, in the situation of Theorem \ref{recthm}, the well-definition of the sequence \eqref{reciprocitysequence} at the second arrow is given by Lemma \ref{surjectivityoftransfer}.

\section{Describing our method and ensuring its efficacy}\label{methodology}

Our goal is to define Campana Points and Darmon Points over $F/K$, where $K$ is a number field and $F$ is an algebraic function field in one variable over $K$. Our method will heavily rely on the local-to-global properties of Pfister forms (see Section \ref{reciprocitypfister}), for which we will require some hypotheses of regularity, linkage, and also the non-hyperbolicity of some Pfister forms. It is important to note that these hypotheses are not all met for every algebraic function field in one variable over a number field, therefore our first objective is to be able to reduce our problem to the case in which the hypotheses are actually met, the main ingredient for such a task being the following result proven in \mbox{\cite[Lemma 5.6]{daans2024universallydefiningsubringsfunction}:}

\begin{theorem}\label{reductionstep}Let $K$ be a field such that $\operatorname{Char}\left(K\right)\neq 2$ and the following hold:

\begin{itemize}

\item $d\coloneqq \operatorname{cd}_2\left(K\right)\in\mathbb{Z}_{\geq 1}$, and there exists a finite field extension $E$ of $K$ with $WG^d\left(E\right)$ nontrivial.

\item If $F$ is any algebraic function field in one variable over $K$, we have that $\left(d+1\right)$-fold quadratic Pfister forms are linked over $F$.

\end{itemize}

Let $F$ be an algebraic function field in one variable over $K$. Then there exists a finite extension $L$ of $K$ contained in the algebraic closure of $F$ with $WG^d\left(L\right)$ nontrivial, and such that $FL$ is a regular algebraic function field in one variable over $L$, and $\left(d+1\right)$-fold quadratic Pfister forms are linked over $FL$.
 
\end{theorem}

Moreover, a weaker version of \cite[Example 5.4(3)]{daans2024universallydefiningsubringsfunction} can be stated as follows:

\begin{proposition}\label{nonrealnumbfieldsarek3}If $K$ is a non-real number field, then $\operatorname{cd}_2\left(K\right)=2$, $WG^2\left(K\right)$ is nontrivial, and if $F$ is any algebraic function field in one variable over $K$, we have that $3$-fold quadratic Pfister forms over $F$ are linked.
\end{proposition}

This is the method we will follow: given a number field $K$ and an algebraic function field $F$ in one variable over $K$, to give a first-order definition of Campana Points and Darmon Points on $F/K$, we will apply Theorem \ref{reductionstep} to the field $K\left(\sqrt{-1}\right)$, which by Proposition \ref{nonrealnumbfieldsarek3} satisfies the hypotheses of the theorem. We will then obtain a finite extension $L$ of $K\left(\sqrt{-1}\right)$ contained in the algebraic closure of $FK\left(\sqrt{-1}\right)$ with $WG^2\left(L\right)$ nontrivial, and such that $FL$ is a regular algebraic function field in one variable over $L$ and $3$-fold quadratic Pfister forms over $FL$ are linked. Under these hypotheses, we will be able to give first-order definitions of Campana points and Darmon points in $FL/L$.

In order for this to be an appropriate approach, we need:

\begin{itemize}

\item An interpretation argument that allows us to transfer those first-order definitions over $FL$ to first-order definitions over $F$ when intersecting our sets with $F$.

\item Given that Campana points and Darmon points are given in terms of valuation conditions, we need an additional argument that proves that intersecting Campana or Darmon points in $FL/L$ with $F$ gives as a result Campana (respectively, Darmon) points in $F/K$, for which we will need valuations in $F/K$ to be unramified in $FL/L$.

\end{itemize}

We now proceed to prove that both needs can be fulfilled.

\subsection{Coordinate polynomials and diophantine restrictions on finite extensions}\label{coopolynomials}

For the first need; namely, the interpretation argument, we will use the theory of coordinate polynomials. For more on this we refer the reader to \cite[§B.7]{MR2297245}.

Given a field $K$ and a field extension $L/K$ of degree $n\in\mathbb{Z}_{\geq 1}$, we fix an ordered $K$-basis $\omega=\left(\omega_1,\cdots,\omega_n\right)\in L^n$ of $L/K$ and we let $\sigma=\left(\sigma_1,\cdots,\sigma_n\right)\in \left(L^\ast\right)^n$ be its corresponding ordered dual basis. We then have\begin{equation}\label{lambdaintermsofbasis}\lambda=\sum_{j=1}^n\sigma_j\left(\lambda\right)\omega_j\text{ for every $\lambda\in L$}\end{equation}and the function $\sigma:L\to K^n$ is bijective, its inverse being $\tau:K^n\to L$ defined as $\tau\left(a\right)\coloneqq\displaystyle{\sum_{j=1}^na_j\omega_j}$ for every $a=\left(a_1,\cdots,a_n\right)\in K^n$.

Given $m\in\mathbb{Z}_{\geq 1}$ and a polynomial $P\in L\left[X_1,\cdots,X_m\right]$, for every function $f:L\to K$ we define \mbox{$f\left(P\right)\in K\left[X_1,\cdots,X_m\right]$} to be the polynomial with coefficients in $K$ which results from applying $f$ to each of the coefficients of $P$. Observe that, rewriting each coefficient of $P$ by using \eqref{lambdaintermsofbasis}, we get\[P\left(\vec{x}\right)=\sum_{j=1}^n\sigma_j\left(P\right)\left(\vec{x}\right)\omega_j\text{ for every $\vec{x}\in L^m$},\]and therefore, if $\vec{x}\in K^m$ then $\sigma_j\left(P\right)\left(\vec{x}\right)\in K$ for every $j\in\left[1,n\right]\cap\mathbb{Z}$, which gives\begin{equation}\label{commevsigma}\sigma_j\left(P\left(\vec{x}\right)\right)=\sigma_j\left(P\right)\left(\vec{x}\right)\text{ for every $\vec{x}\in K^m$ and $j\in\left[1,n\right]\cap\mathbb{Z}$}.\end{equation}

\begin{lemma}\label{restricteddiophdef}Let $K$ be a field, $L$ a finite field extension of $K$ of degree $n\in\mathbb{Z}_{\geq 1}$, let $\omega\in L^n$ be a $K$-ordered basis of $L/K$ and let $\sigma\in\left(L^\ast\right)^n$ be its corresponding ordered dual basis.

Fix $m,r\in\mathbb{Z}_{\geq 1}$, $\vec{x}\in K^m$ and a polynomial $P\in L\left[X_1,\cdots,X_m,Y_1,\cdots Y_r\right]$. For each \mbox{$k\in\left[1,n\right]\cap\mathbb{Z}$} define\begin{multline*}P_k\coloneqq \sigma_k\left[P\left(\left(X_i\right)_{i\in\left[1,m\right]\cap\mathbb{Z}},\left(\sum_{j=1}^nY_{ij}\omega_j\right)_{i\in\left[1,r\right]\cap\mathbb{Z}}\right)\right]\in K\left[\left(X_i\right)_{i\in\left[1,m\right]\cap\mathbb{Z}},\left(Y_{ij}\right)_{\substack{i\in\left[1,r\right]\cap\mathbb{Z}\\j\in\left[1,n\right]\cap\mathbb{Z}}}\right]\\=\sigma_k\left[P\left(X_1,\cdots,X_m,\sum_{j=1}^nY_{1j}\omega_j,\cdots,\sum_{j=1}^nY_{rj}\omega_j\right)\right]\in K\left[X_1,\cdots,X_m,Y_{11},\cdots, Y_{rn}\right].\end{multline*}

The following are equivalent:

\begin{enumerate}[(i)]

\item There exists $\vec{y}\in L^r$ such that $P\left(\vec{x},\vec{y}\right)=0$.

\item There exists $\mathcal{Y}\in K^{r\times n}$ such that for every $k\in\left[1,n\right]\cap\mathbb{Z}$ we have $P_k\left(\vec{x},\mathcal{Y}\right)=0$.

\end{enumerate}
\end{lemma}

\begin{proof}If there exists $\vec{y}=\left(y_1,\cdots,y_r\right)\in L^r$ such that $P\left(\vec{x},\vec{y}\right)=0$, define $\mathcal{Y}_{ij}\coloneqq\sigma_j\left(y_i\right)$ for every $i\in\left[1,r\right]\cap\mathbb{Z}$ and $j\in\left[1,n\right]\cap\mathbb{Z}$, so that \eqref{lambdaintermsofbasis} gives $\vec{y}=\displaystyle{\left(\sum_{j=1}^n\mathcal{Y}_{ij}\omega_j\right)_{i\in\left[1,r\right]\cap\mathbb{Z}}}$, therefore\begin{multline}\label{p(vecx,vecy)=sum(pk(vecx,matrixy))}P\left(\vec{x},\vec{y}\right)=P\left(\vec{x},\left(\sum_{j=1}^n\mathcal{Y}_{ij}\omega_j\right)_{i\in\left[1,r\right]\cap\mathbb{Z}}\right)\\=\sum_{k=1}^n\sigma_k\left[P\left(\vec{x},\left(\sum_{j=1}^n\mathcal{Y}_{ij}\omega_j\right)_{i\in\left[1,r\right]\cap\mathbb{Z}}\right)\right]\omega_k=\sum_{k=1}^nP_k\left(\vec{x},\mathcal{Y}\right)\omega_k,\end{multline}where in the last equality we have used \eqref{commevsigma}, which can be applied because $\vec{x}\in K^m$ by hypothesis and $\mathcal{Y}\in K^{r\times n}$ by definition. We conclude that $P_k\left(\vec{x},\mathcal{Y}\right)=0$ for every $k\in\left[1,n\right]\cap\mathbb{Z}$.

Conversely, assume there exists $\mathcal{Y}\in K^{r\times n}$ such that for every $k\in\left[1,n\right]\cap\mathbb{Z}$ we have $P_k\left(\vec{x},\mathcal{Y}\right)=0$. For each $i\in\left[1,r\right]\cap\mathbb{Z}$ define $\vec{y}\coloneqq\displaystyle{\left(\sum_{j=1}^n\mathcal{Y}_{ij}\omega_j\right)_{i\in\left[1,r\right]\cap\mathbb{Z}}\in L^r}$. We then have\[0=\sum_{k=1}^nP_k\left(\vec{x},\mathcal{Y}\right)\omega_k=P\left(\vec{x},\vec{y}\right),\]where for the last equality we have inverted the algebraic procedure in \eqref{p(vecx,vecy)=sum(pk(vecx,matrixy))}. $\blacksquare$

\end{proof}

Lemma \ref{restricteddiophdef} has the following immediate corollary, which allows us to restrict first-order definitions to base fields on finite extensions, and is a generalized version of \cite[Lemma 6.3]{MILLER2022103076}.

\begin{corollary}\label{backandforth}Let $K$ be a non-algebraically closed field and let $L$ be a finite extension of $K$. Fix $m\in\mathbb{Z}_{\geq 1}$ and $D\subseteq L^m$.

\begin{enumerate}[(i)]

\item If $D$ is definable by an existential formula over $L$, then $D\cap K^m$ is definable by an existential formula over $K$.

\item If $D$ is definable by a universal formula over $L$, then $D\cap K^m$ is definable by a universal formula over $K$.

\item If $D$ is $\forall\exists$-definable over $L$, then $D\cap K^m$ is $\forall\exists$-definable over $K$.

\item If $D$ is $\forall\exists\forall$-definable over $L$, then $D\cap K^m$ is $\forall\exists\forall$-definable over $K$.

\end{enumerate}

\end{corollary}

\begin{proof}If $D$ is existentially defined over $L$, there exist $r\in\mathbb{Z}_{\geq 1}$ and $P\in L\left[X_1,\cdots,X_m,Y_1,\cdots Y_r\right]$ such that $D$ is defined by the formula\[\exists y_1\cdots\exists y_r\left(P\left(x_1,\cdots,x_m,y_1,\cdots,y_r\right)=0\right).\]If $n\coloneqq\left[L:K\right]$, for each $k\in\left[1,n\right]\cap\mathbb{Z}$ define $P_k$ as in Lemma \ref{restricteddiophdef}.

Given $\vec{x}\in K^m$, we then have $\vec{x}\in D$ if and only if there exists $\vec{y}\in L^r$ such that $P\left(\vec{x},\vec{y}\right)=0$, which by Lemma \ref{restricteddiophdef} is equivalent to the existence of $\mathcal{Y}\in K^{r\times n}$ such that for every $k\in\left[1,n\right]\cap\mathbb{Z}$ we have $P_k\left(\vec{x},\mathcal{Y}\right)=0$. Since $K$ is not algebraically closed, \cite[Lemma 1.2.3]{MR2297245} says that there exists $Q\in K\left[X_1,\cdots,X_m,Y_{11},\cdots, Y_{rn}\right]$ which vanishes exactly at the common solutions of $\left\{P_k=0:k\in\left[1,n\right]\cap\mathbb{Z}\right\}$. We conclude that $D\cap K^m$ is definable over $K$ with the formula\[\exists y_{11}\cdots\exists y_{rn}\left(Q\left(x_1,\cdots,x_m,y_{11},\cdots,y_{rn}\right)=0\right),\]and this proves $\left(i\right)$.

To prove $\left(ii\right)$, assume $D$ is universal over $L$. Then $L^m\setminus D$ is existential over $L$, which by $\left(i\right)$ shows that $K^m\cap \left(L^m\setminus D\right)=K^m\setminus D=K^m\setminus \left(D\cap K^m\right)$ is existential over $K$, and therefore $D\cap K^m$ is universal over $K$.

Items $\left(iii\right)$ and $\left(iv\right)$, as well as any other statement of the sort, follow from iteratively combining the procedure to prove items $\left(i\right)$ and $\left(ii\right)$, for various values of $m\in\mathbb{Z}_{\geq 1}$. $\blacksquare$

\end{proof}

\subsection{Non-ramification of places by constant extensions}\label{nonramification}

The second issue we needed to study in order for our proposed approach to be effective were the difficulties that ramified places could potentially produce when intersecting Campana and Darmon points with lower-level algebraic function fields. To be concrete, assume we have a number field $K$, an algebraic function field $F$ in one variable over $K$, a finite extension $L$ of $K$ contained in the algebraic closure of $F$, and a fixed $n\in\mathbb{Z}_{\geq 1}$. Campana points over $FL/L$ are defined in terms of conditions of the form $w\left(x\right)\in\mathbb{Z}_{\geq 0}\cup\mathbb{Z}_{\leq -n}$ for various places $w\in\mathcal{V}\left(FL/L\right)$, while Darmon points over $FL/L$ are defined in terms of conditions of the form $w\left(x\right)\in\mathbb{Z}_{\geq 0}\cup n\mathbb{Z}$ for various places $w\in\mathcal{V}\left(FL/L\right)$. Fixing such a $w$, these conditions applied to $x\in F$ become $e\left(w\mid v\right)v\left(x\right)\in\mathbb{Z}_{\geq 0}\cup \mathbb{Z}_{\leq -n}$ and $e\left(w\mid v\right)v\left(x\right)\in\mathbb{Z}_{\geq 0}\cup n\mathbb{Z}$, respectively, where $v\in\mathcal{V}\left(F/K\right)$ lies under $w$. As we can see, ramification indexes are potentially dangerous for our desired interpretation argument.

In this section we will show that no such difficulty will arise in our context, as our ramification indexes will always equal $1$. We start with a sufficient criterion for non-ramification:

\begin{lemma}\label{unramifiedbyextension}Let $v$ be a discrete valuation in a field $K$, and let $f\in\mathcal{O}_v\left[x\right]$ be a monic irreducible polynomial whose reduction modulo $\mathfrak{m}_v$ is separable in $Kv\left[x\right]$. Then $v$ is unramified in $K\left[x\right]/\left(f\right)$.
\end{lemma}

\begin{proof}See \cite[Lemma 2.3.4]{Fried2023}. $\blacksquare$
\end{proof}

This criterion will immediately give the non-ramification result we need.

\begin{lemma}\label{unramifiedbyextension2}Let $F$ be an algebraic function field in one variable over a number field $K$, $L$ a finite extension of $K$ contained in the algebraic closure of $F$, $v\in\mathcal{V}\left(F/K\right)$, and $w\in\mathcal{V}\left(FL/L\right)$ lying over $v$. Then $e\left(w\mid v\right)=1$.
\end{lemma}

\begin{proof}Since $L/K$ is finite and separable because $\operatorname{Char}\left(K\right)=0$, by the primitive element theorem (\cite[Chapter V, Theorem 4.6]{MR1878556}) there exists $\alpha\in L$ such that $L=K\left(\alpha\right)$, so that $FL=F\left(\alpha\right)$. Let $f\in K\left[x\right]\subseteq \mathcal{O}_v\left[x\right]$ be the minimal polynomial of $\alpha$ over $K$. Pick a monic irreducible factor $g$ of $f$ in $\mathcal{O}_v\left[x\right]$ having $\alpha$ as a root, so that $FL=F\left(\alpha\right)\cong F\left[x\right]/\left(g\right)$. Observe that the reduction of $g$ modulo $\mathfrak{m}_v$ is separable because $K$ is a characteristic-zero subfield of $Fv$, so Lemma \ref{unramifiedbyextension} gives the desired conclusion. $\blacksquare$\end{proof}

\section{First-order definitions}\label{setupandproofs}

We are now ready to provide our desired first-order definitions. We start by introducing some notation.

\begin{definition}Let $K$ be a field and let $F$ be an algebraic function field in one variable over $K$. Let $n\in\mathbb{Z}_{\geq 1}$, let $b,c\in F$ be such that $c\left(1+4b\right)\neq 0$, and let $a_1,\cdots,a_n\in F^\times$. We define\begin{align*}\Delta_{a_1,\cdots,a_n,b,F/K}&\coloneqq\left\{v\in\mathcal{V}\left(F/K\right):\left[\llangle a_1,\cdots,a_n,b]]_{F}\right]+WG^{n+2}\left(F\right)\not\in\operatorname{ker}\left(\partial_v^n\right)\right\},\\\Delta_c^{a_1,\cdots,a_n,b,F/K}&\coloneqq\Delta_{a_1,\cdots,a_n,b,F/K}\cap\left\{v\in\mathcal{V}\left(F/K\right):2\nmid v\left(c\right)\right\},\\T_{a_1,\cdots,a_n,b,F/K}&\coloneqq \bigcap_{v\in\Delta_{a_1,\cdots,a_n,b,F/K}}\mathcal{O}_v,\\J_{c}\left(a_1,\cdots,a_n,b,F/K\right)&\coloneqq \bigcap_{v\in\Delta_c^{a_1,\cdots,a_n,b,F/K}}\mathfrak{m}_v.\end{align*}
\end{definition}

We begin by showing that these definitions allow us to parametrize finite subsets of $\mathbb{Z}$-valuations.

\begin{proposition}\label{paramplaces}Let $K$ be a field with $\operatorname{Char}\left(K\right)\neq 2$, \mbox{$d\coloneqq\operatorname{cd}_2\left(K\right)\in\mathbb{Z}_{\geq 1}$,} and $WG^d\left(K\right)$ is nontrivial. Let $F$ be a regular algebraic function field in one variable over $K$ such that \mbox{$\left(d+1\right)$-fold} quadratic Pfister forms are linked over $F$, and let $S$ be a finite subset of $\mathcal{V}\left(F/K\right)$ of even cardinality. Then there exist \mbox{$a_1,\cdots,a_d,b,c\in \displaystyle{\bigcap_{v\in S}\mathcal{O}_v}$} such that $v\left(c\right)=1$ for all $v\in S$, $a_n\in F^\times$, $a_1,\cdots,a_{d-1},1+4b\in\displaystyle{\bigcap_{v\in S}\mathcal{O}_v^\times}$, and\[S=\Delta_{a_1,\cdots,a_n,b,F/K}=\Delta_c^{a_1,\cdots,a_n,b,F/K}.\]
\end{proposition}

\begin{proof}Fix a nonzero $\alpha\in WG^d\left(K\right)$. For each $v\in S$, by Lemma \ref{surjectivityoftransfer} there exists $\beta_v\in WG^d\left(Fv\right)$ such that $s_{Fv/K}\left(\beta_v\right)=\alpha$. Observe that $\beta_v$ is nonzero since $\alpha$ is nonzero. Define\[\beta\coloneqq \sum_{v\in S}\beta_v\in\bigoplus_{v\in\mathcal{V}\left(F/K\right)}WG^d\left(Fv\right),\]so that the image of $\beta$ under the second arrow of \eqref{reciprocitysequence} equals $\left|S\right|\alpha$. Proposition \ref{dim2k} gives that $WG^{d+1}\left(K\right)$ is trivial, so by Lemma \ref{2-torsion} the group $WG^{d}\left(K\right)$ is $2$-torsion, which together with the fact that $\left|S\right|$ is even gives that $\beta$ belongs to the kernel of the second arrow of \eqref{reciprocitysequence}. By the exactness of the sequence given by Theorem \ref{recthm} there exists $\gamma\in WG^{d+1}\left(F\right)$ such that:\[\text{If $v\in\mathcal{V}\left(F/K\right)$, then }\partial_v^n\left(\gamma\right)=\begin{cases}\beta_v,&v\in S,\\0,&v\not\in S.\end{cases}\]Since $\left(d+1\right)$-fold quadratic Pfister forms are linked over $F$ then $\gamma=\left[q\right]$ for some \mbox{$\left(d+1\right)$-fold} quadratic Pfister form over $F$. Lemma \ref{rewrparampfister} gives $q\cong\llangle a_1,\cdots,a_d,b]]_F$ for some \mbox{$a_1,\cdots,a_d,b\in \displaystyle{\bigcap_{v\in S}\mathcal{O}_v}$} such that $a_d\neq 0$ and $a_1,\cdots,a_{d-1},1+4b\in\displaystyle{\bigcap_{v\in S}\mathcal{O}_v^\times}$. The fact that $\beta_v$ is nonzero for each $v\in S$ implies $S=\Delta_{a_1,\cdots,a_n,b,F/K}$.

For each $v\in S$ fix $\pi_v\in F$ with $v\left(\pi_v\right)=1$, and use the Chinese Remainder Theorem to find $c\in \mathcal{O}_{F/K}$ satisfying $c\equiv \pi_v\pmod{\mathfrak{m}_v^2}$ for all $v\in S$. We therefore get $v\left(c\right)=1$ for all $v\in S$, giving $S\subseteq \left\{v\in\mathcal{V}\left(F/K\right):2\nmid v\left(c\right)\right\}$ and thus $S=\Delta_{a_1,\cdots,a_n,b,F/K}=\Delta_c^{a_1,\cdots,a_n,b,F/K}$. $\blacksquare$
\end{proof}

Combining Proposition \ref{paramplaces} with Proposition \ref{nonrealnumbfieldsarek3} we immediately obtain:

\begin{proposition}\label{paramplaces2}Let $K$ be a non-real number field and let $F$ be a regular algebraic function field in one variable over $K$. If $S$ is a finite subset of $\mathcal{V}\left(F/K\right)$ of even cardinality, there exist $a_1,a_2,b,c\in\displaystyle{\bigcap_{v\in S}\mathcal{O}_v}$ such that $a_2\in F^\times$, $a_1,1+4b\in\displaystyle{\bigcap_{v\in S}\mathcal{O}_v^\times}$, $v\left(c\right)=1$ for all $v\in S$, and\[S=\Delta_{a_1,a_2,b,F/K}=\Delta_c^{a_1,a_2,b,F/K}.\]
\end{proposition}

\subsection{Some preliminary uniform first-order definitions}

To first-order define Campana points and Darmon points in our desired context, we will make use of some already known first-order descriptions. These are:

\begin{proposition}\label{basicdefinitions}Let $K$ be a non-real number field and let $F$ be an algebraic function field in one variable over $K$. The sets\begin{align*}\left\{\left(a_1,a_2,b,x\right)\in F^4\right.&\left.:a_1a_2\left(1+4b\right)\neq 0\wedge x\in T_{a_1,a_2,b,F/K}\right\},\\\left\{\left(a_1,a_2,b,c,x\right)\in F^5\right.&\left.:a_1a_2c\left(1+4b\right)\neq 0\wedge x\in J_c\left(a_1,a_2,b,F/K\right)\right\}\end{align*}are diophantine over $F$.
\end{proposition}

\begin{proof}The assertion regarding the first set is a particular case of \cite[Theorem 10.13]{becher2025uniformexistentialdefinitionsvaluations}, while the second assertion follows directly from the first one and the following equality, which is a direct consequence of \cite[Lemma 5.4]{MR4378716}:\[J_c\left(a_1,a_2,b,F/K\right)=\left(c\cdot F^{2}\cap \left(1-F^{2}\cdot T_{a_1,a_2,b,F/K}^\times\right)\right)\cdot T_{a_1,a_2,b,F/K}\]for all $a_1,a_2,b,c\in F$ with $a_1a_2c\left(1+4b\right)\neq 0$. $\blacksquare$
\end{proof}

We introduce another notation: if $K$ is a non-real number field and $F$ is an algebraic function field in one variable over $K$, observe that if $S$ and $T$ are two finite subsets of $\mathcal{V}\left(F/K\right)$ then\begin{equation}\label{ja+jb=jab}\bigcap_{v\in S}\mathfrak{m}_v+\bigcap_{v\in T}\mathfrak{m}_v=\bigcap_{v\in S\cap T}\mathfrak{m}_v,\end{equation}so that $S\cap T=\emptyset$ if and only if $1\in\displaystyle{\bigcap_{v\in S}\mathfrak{m}_v+\bigcap_{v\in T}\mathfrak{m}_v}$.

For all $a_1,a_2,b,c,a_1',a_2',b',c'\in F$ with $a_1a_2c\left(1+4b\right)a_1'a_2'c'\left(1+4b'\right)\neq 0$ we define\begin{align*}\Omega_{c,c'}^{a_1,a_2,b,a_1',a_2',b',F/K}&\coloneqq \Delta_c^{a_1,a_2,b,F/K}\cap \Delta_{c'}^{a_1',a_2',b',F/K},\\J_{c,c'}\left(a_1,a_2,b,a_1',a_2',b',F/K\right)&\coloneqq J_c\left(a_1,a_2,b,F/K\right)+J_{c'}\left(a_1',a_2',b',F/K\right)\\&=\bigcap_{v\in\Omega_{c,c'}^{a_1,a_2,b,a_1',a_2',b',F/K}}\mathfrak{m}_v.\end{align*}Recalling that the formula $x\neq 0$ can be existentially defined as $\exists y\left(xy-1=0\right)$ and that any finite number of polynomial equations over $K$ can be reduced to a single one by \mbox{\cite[Lemma 1.2.3]{MR2297245},} then by Proposition \ref{basicdefinitions} the formula\begin{multline*}a_1a_2c\left(1+4b\right)a_1'a_2'c'\left(1+4b'\right)\neq 0\\\wedge \exists y\left(y\in J_c\left(a_1,a_2,b,F/K\right)\wedge x-y\in J_{c'}\left(a_1',a_2',b',F/K\right)\right)\end{multline*}proves the following result:

\begin{lemma}\label{jacobson}Let $K$ be a non-real number field and let $F$ be an algebraic function field in one variable over $K$. The set\begin{multline*}\left\{\left(a_1,a_2,b,c,a_1',a_2',b',c',x\right)\in F^9:a_1a_2c\left(1+4b\right)a_1'a_2'c'\left(1+4b'\right)\neq 0\right.\\\left.\wedge \,x\in J_{c,c'}\left(a_1,a_2,b,a_1',a_2',b',F/K\right)\right\}\end{multline*}is diophantine over $F$.
\end{lemma}

Also, a direct consequence of Proposition \ref{paramplaces2} is:

\begin{proposition}\label{paramplaces3}Let $K$ be a non-real number field and let $F$ be a regular algebraic function field in one variable over $K$. If $S$ is a finite subset of $\mathcal{V}\left(F/K\right)$, there exist $a_1,a_2,b,c,a_1',a_2',b',c'\in F$ such that $a_1a_2c\left(1+4b\right)a_1'a_2'c'\left(1+4b'\right)\neq 0$ and $S=\Omega_{c,c'}^{a_1,a_2,b,a_1',a_2',b',F/K}$.
\end{proposition}

Another important first-order definable notion in this context is the empty intersection between two subsets of $\mathbb{Z}$-valuations that are trivial on the base field. We now prove this to be a diophantine statement.

\begin{lemma}\label{emptyint}Let $K$ be a non-real number field and let $F$ be an algebraic function field in one variable over $K$. The set\begin{multline*}\left\{\left(a_1,a_2,b,c,a_1',a_2',b',c',\alpha_1,\alpha_2,\beta,\gamma,\alpha_1',\alpha_2',\beta',\gamma'\right)\in F^{16}:\right.\\\left.a_1a_2c\left(1+4b\right)a_1'a_2'c'\left(1+4b'\right)\alpha_1\alpha_2\gamma\left(1+4\beta\right)\alpha_1'\alpha_2'\gamma'\left(1+4\beta'\right)\neq 0\right.\\\left. \wedge\, \Omega_{c,c'}^{a_1,a_2,b,a_1',a_2',b',F/K}\cap \Omega_{\gamma,\gamma'}^{\alpha_1,\alpha_2,\beta,\alpha_1',\alpha_2',\beta',F/K}=\emptyset\right\}\end{multline*}is diophantine.
\end{lemma}

\begin{proof}Follows directly from \eqref{ja+jb=jab}, Proposition \ref{basicdefinitions}, and the formula\begin{multline*}a_1a_2c\left(1+4b\right)a_1'a_2'c'\left(1+4b'\right)\alpha_1\alpha_2\gamma\left(1+4\beta\right)\alpha_1'\alpha_2'\gamma'\left(1+4\beta'\right)\neq 0\\\wedge 1\in \left[J_{c,c'}\left(a_1,a_2,b,a_1',a_2',b',F/K\right)+J_{\gamma,\gamma'}\left(\alpha_1,\alpha_2,\beta,\alpha_1',\alpha_2',\beta',F/K\right)\right].\text{ }\blacksquare\end{multline*}
\end{proof}

In the above proof, if we change $J_{\gamma,\gamma'}\left(\alpha_1,\alpha_2,\beta,\alpha_1',\alpha_2',\beta',F/K\right)$ by $J_{\gamma}\left(\alpha_1,\alpha_2,\beta,F/K\right)$ and get rid of the variables $\alpha_1',\alpha_2',\beta',\gamma'$, we immediately get:

\begin{lemma}\label{emptyint2}Let $K$ be a non-real number field and let $F$ be an algebraic function field in one variable over $K$. The set\begin{multline*}\left\{\left(a_1,a_2,b,c,a_1',a_2',b',c',\alpha_1,\alpha_2,\beta,\gamma\right)\in F^{12}:\right.\\\left.a_1a_2c\left(1+4b\right)a_1'a_2'c'\left(1+4b'\right)\alpha_1\alpha_2\gamma\left(1+4\beta\right)\neq 0\right.\\\left. \wedge\, \Omega_{c,c'}^{a_1,a_2,b,a_1',a_2',b',F/K}\cap \Delta_{\gamma}^{\alpha_1,\alpha_2,\beta,F/K}=\emptyset\right\}\end{multline*}is diophantine.
\end{lemma}

\begin{lemma}\label{equalitytwodeltas}Let $K$ be a non-real number field and let $F$ be an algebraic function field in one variable over $K$. The set\[\left\{\left(a_1,a_2,b,c\right)\in F^{4}:a_1a_2c\left(1+4b\right)\neq 0\wedge\,\Delta_{a_1,a_2,b,F/K}=\Delta_c^{a_1,a_2,b,F/K}\right\}\]is $\forall\exists$-definable in $F$.
\end{lemma}

\begin{proof}We claim that the set is defined by\[a_1a_2c\left(1+4b\right)\neq 0\wedge \forall x\left(x\in J_c\left(a_1,a_2,b,F/K\right)\Rightarrow x\in T_{a_1,a_2,b,F/K}\right),\]which is a $\forall\exists$-formula by Proposition \ref{basicdefinitions}.

Given $\left(a_1,a_2,b,c\right)\in F^{4}$ with $a_1a_2c\left(1+4b\right)\neq 0$, we want to show that the sets $\Delta_{a_1,a_2,b,F/K}$ and $\Delta_c^{a_1,a_2,b,F/K}$ are equal if and only if $J_c\left(a_1,a_2,b,F/K\right)\subseteq T_{a_1,a_2,b,F/K}$.

Indeed, if $\Delta_{a_1,a_2,b,F/K}=\Delta_c^{a_1,a_2,b,F/K}$ then $J_c\left(a_1,a_2,b,F/K\right)$ is the Jacobson radical of $T_{a_1,a_2,b,F/K}$, thus the desired inclusion is immediate. Conversely, if $\Delta_{a_1,a_2,b,F/K}\neq\Delta_c^{a_1,a_2,b,F/K}$ then, since $\Delta_c^{a_1,a_2,b,F/K}\subseteq \Delta_{a_1,a_2,b,F/K}$, there exists $v\in \Delta_{a_1,a_2,b,F/K}\setminus \Delta_c^{a_1,a_2,b,F/K}$. By Weak Approximation there exists $x\in F$ with $v\left(x\right)<0$ and $w\left(x\right)>0$ for all $w\in \Delta_c^{a_1,a_2,b,F/K}$. The former implies $x\not\in T_{a_1,a_2,b,F/K}$, while the latter implies $x\in J_c\left(a_1,a_2,b,F/K\right)$, thus $J_c\left(a_1,a_2,b,F/K\right)\not\subseteq T_{a_1,a_2,b,F/K}$. $\blacksquare$
\end{proof}

With all these results, we can now proceed to give our first-order definitions for Campana points and Darmon points in algebraic function fields in one variable over number fields, by following the method described in Section \ref{methodology}.

\subsection{Campana points}\label{provingcampana}

Without any further delay, let us state and prove the first of our two main results.

\begin{theorem}\label{mainthmcampana}Let $K$ be a number field, let $F$ be an algebraic function field in one variable over $K$, let $S$ be a finite subset of $\mathcal{V}\left(F/K\right)$, and let $n\in\mathbb{Z}_{\geq 1}$. Then the set $C_{F/K,S,n}$ is $\forall\exists$-definable in $F$, uniformly with respect to all choices of $S$.
\end{theorem}

\begin{proof}Since $K\left(\sqrt{-1}\right)$ is a non-real number field, Proposition \ref{nonrealnumbfieldsarek3} gives $\operatorname{cd}_2\left(K\left(\sqrt{-1}\right)\right)=2$, $WG^2\left(K\left(\sqrt{-1}\right)\right)$ is nontrivial, and $3$-fold quadratic Pfister forms over any algebraic function field in one variable over $K\left(\sqrt{-1}\right)$ are linked. It follows from Theorem \ref{reductionstep} that there exists a finite extension $L$ of $K\left(\sqrt{-1}\right)$ contained in the algebraic closure of $FK\left(\sqrt{-1}\right)$ such that $F'\coloneqq FL$ is a regular algebraic function field in one variable over $L$. Observe that $L$ is also a non-real number field. Define\[S'\coloneqq\left\{w\in\mathcal{V}\left(F'/L\right):\exists v\in S\left(\mathcal{O}_w\cap K=\mathcal{O}_v\right)\right\}\]as the set of places in $\mathcal{V}\left(F'/L\right)$ lying over some place in $S$. By Proposition \ref{paramplaces3} there are $a_1,a_2,b,c,a_1',a_2',b',c'\in F'$ with $a_1a_2c\left(1+4b\right)a_1'a_2'c'\left(1+4b'\right)\neq 0$ and $S'=\Omega_{c,c'}^{a_1,a_2,b,a_1',a_2',b',F'/L}$.

We claim that $C_{F'/L,S',n}$ is defined in $F'$ through the formula (in the free variable $x$)\begin{multline}\label{campanadeffstorder}\forall\alpha_1\forall\alpha_2\forall\beta\forall\gamma\forall\alpha'_1\forall\alpha'_2\forall\beta'\forall\gamma'\\\left[\begin{pmatrix}\alpha_1\alpha_2\gamma\left(1+4\beta\right)\alpha'_1\alpha'_2\gamma'\left(1+4\beta'\right)\neq 0\\\Omega_{c,c'}^{a_1,a_2,b,a_1',a_2',b',F'/L}\cap \Omega_{\gamma,\gamma'}^{\alpha_1,\alpha_2,\beta,\alpha_1',\alpha_2',\beta',F'/L}=\emptyset\\x\in \left(J_{\gamma,\gamma'}\left(\alpha_1,\alpha_2,\beta,\alpha'_1,\alpha'_2,\beta',F'/L\right)\setminus\left\{0\right\}\right)^{-1}\end{pmatrix}\Rightarrow \right.\\\left.x\in \left(J_{\gamma,\gamma'}\left(\alpha_1,\alpha_2,\beta,\alpha'_1,\alpha'_2,\beta',F'/L\right)\setminus\left\{0\right\}\right)^{-n}\right].\end{multline}Observe that Lemma \ref{emptyint} and Lemma \ref{jacobson} show that this is indeed a $\forall\exists$-formula. Let us show that it defines $C_{F'/L,S',n}$.

Given $x\in F'$ such that \eqref{campanadeffstorder} holds on $x$ and $w\in\mathcal{V}\left(F'/L\right)\setminus S'$ such that \mbox{$w\left(x\right)\not\in\mathbb{Z}_{\geq 0}$}, we want to show that $w\left(x\right)\leq -n$. By Proposition \ref{paramplaces3} there exist $\alpha_1,\alpha_2,\beta,\gamma,\alpha_1',\alpha_2',\beta',\gamma'\in F'$ such that $\alpha_1\alpha_2\gamma\left(1+4\beta\right)\alpha'_1\alpha'_2\gamma'\left(1+4\beta'\right)\neq 0$ and $\Omega_{\gamma,\gamma'}^{\alpha_1,\alpha_2,\beta,\alpha_1',\alpha_2',\beta',F'/L}=\left\{w\right\}$. Since $w\not\in S'$ then $\Omega_{c,c'}^{a_1,a_2,b,a_1',a_2',b',F'/L}\cap \Omega_{\gamma,\gamma'}^{\alpha_1,\alpha_2,\beta,\alpha_1',\alpha_2',\beta',F'/L}=\emptyset$, and since $w\left(x\right)<0$ then $x\in \left(\mathfrak{m}_w\setminus\left\{0\right\}\right)^{-1}=\left(J_{\gamma,\gamma'}\left(\alpha_1,\alpha_2,\beta,\alpha'_1,\alpha'_2,\beta',F'/L\right)\setminus\left\{0\right\}\right)^{-1}$, hence \eqref{campanadeffstorder} implies\[x\in\left(J_{\gamma,\gamma'}\left(\alpha_1,\alpha_2,\beta,\alpha'_1,\alpha'_2,\beta',F'/L\right)\setminus\left\{0\right\}\right)^{-n}=\left(\mathfrak{m}_w\setminus\left\{0\right\}\right)^{-n},\]giving $w\left(x\right)\leq -n$, as desired.

Conversely, let $x\in C_{F'/L,S',n}$ and fix $\alpha_1,\alpha_2,\beta,\gamma,\alpha_1',\alpha_2',\beta',\gamma'\in F'$ such that:

\begin{itemize}

\item $\alpha_1\alpha_2\gamma\left(1+4\beta\right)\alpha'_1\alpha'_2\gamma'\left(1+4\beta'\right)\neq 0$,

\item $\Omega_{c,c'}^{a_1,a_2,b,a_1',a_2',b',F'/L}\cap \Omega_{\gamma,\gamma'}^{\alpha_1,\alpha_2,\beta,\alpha_1',\alpha_2',\beta',F'/L}=\emptyset$, and

\item $x\in \left(J_{\gamma,\gamma'}\left(\alpha_1,\alpha_2,\beta,\alpha'_1,\alpha'_2,\beta',F'/L\right)\setminus\left\{0\right\}\right)^{-1}$.

\end{itemize}

If we let $T\coloneqq \Omega_{\gamma,\gamma'}^{\alpha_1,\alpha_2,\beta,\alpha_1',\alpha_2',\beta',F'/L}$, the second condition implies $S'\cap T=\emptyset$, so $w\not\in S'$ for all $w\in T$. The third condition implies $w\left(x\right)<0$ for all $w\in T$.

Since no element of $T$ belongs to $S'$ and $x\in C_{F'/L,S',n}$, then the third condition also implies $w\left(x\right)\leq -n$ for all $w\in T$, which is nothing but $x\in\left(J_{\gamma,\gamma'}\left(\alpha_1,\alpha_2,\beta,\alpha'_1,\alpha'_2,\beta',F'/L\right)\setminus\left\{0\right\}\right)^{-n}$.

Corollary \ref{backandforth} implies that $C_{F'/L,S',n}\cap F$ is $\forall\exists$-definable in $F$. Since all $v\in\mathcal{V}\left(F/K\right)$ lie under some $w\in\mathcal{V}\left(F'/L\right)$, and all $w\in\mathcal{V}\left(F'/L\right)$ lie over some $v\in\mathcal{V}\left(F/K\right)$, then by definition of $S'$ we have that for all $v\in \mathcal{V}\left(F/K\right)\setminus S$ there exists $w\in \mathcal{V}\left(F'/L\right)\setminus S'$ lying over $v$, and for all $w\in \mathcal{V}\left(F'/L\right)\setminus S'$ there exists $v\in \mathcal{V}\left(F/K\right)\setminus S$ lying under $w$. In other words, if for each $w\in \mathcal{V}\left(F'/L\right)\setminus S'$ we let $v_w\in \mathcal{V}\left(F/K\right)\setminus S$ be the place lying under $w$, then we have a well-defined surjection $\mathcal{V}\left(F'/L\right)\setminus S'\to \mathcal{V}\left(F/K\right)\setminus S$ given by $w\mapsto v_w$ for all $w\in \mathcal{V}\left(F'/L\right)\setminus S'$. Now\begin{multline*}C_{F'/L,S',n}\cap F=\left\{x\in F:\forall w\in \mathcal{V}\left(F'/L\right)\setminus S'\left(w\left(x\right)\in\mathbb{Z}_{\geq 0}\cup\mathbb{Z}_{\leq -n}\right)\right\}\\=\left\{x\in F:\forall w\in \mathcal{V}\left(F'/L\right)\setminus S'\left(e\left(w\mid v_w\right)v_w\left(x\right)\in\mathbb{Z}_{\geq 0}\cup\mathbb{Z}_{\leq -n}\right)\right\}\\=\left\{x\in F:\forall w\in \mathcal{V}\left(F'/L\right)\setminus S'\left(v_w\left(x\right)\in\mathbb{Z}_{\geq 0}\cup\mathbb{Z}_{\leq -n}\right)\right\},\end{multline*}where in the last equality we have used Lemma \ref{unramifiedbyextension2}. Finally, since $w\mapsto v_w$ is a surjection $\mathcal{V}\left(F'/L\right)\setminus S'\to \mathcal{V}\left(F/K\right)\setminus S$, we get\[C_{F'/L,S',n}\cap F=\left\{x\in F:\forall v\in \mathcal{V}\left(F/K\right)\setminus S\left(v\left(x\right)\in\mathbb{Z}_{\geq 0}\cup\mathbb{Z}_{\leq -n}\right)\right\}=C_{F/K,S,n},\]which shows that $C_{F/K,S,n}$ is $\forall\exists$-definable in $F$, as desired. $\blacksquare$
\end{proof}

\begin{remark}\label{campanauniform}In the context of Theorem \ref{mainthmcampana}, we have actually proved that $C_{F/K,S,n}$ is \mbox{$\forall\exists$-definable} in a uniform way with respect to all finite $S\subseteq\mathcal{V}\left(F/K\right)$. Indeed, once we have fixed our $L$, the correspondence $S\mapsto S'$ is injective, and the formula in $F'$ is uniform with respect to the parameters that define $S'$ (cf. Proposition \ref{paramplaces3}). To descend to a first-order description in $F$ we used Corollary \ref{backandforth}, and a closer look at its proof shows that it also keeps track of the uniformity: for a fixed ordered $K$-basis $\Omega\subseteq L^{\left[L:K\right]}$ of $L/K$, the uniformity of our formula for $C_{F'/L,S',n}$ with respect to the parameters $a_1,a_2,b,c,a_1',a_2',b',c'$ descenders to a uniformity with respect to the $\Omega$-coordinates of these parameters.
\end{remark}

\subsection{Darmon points}\label{provingdarmon}

We will follow the exact same method as in Section \ref{provingcampana} to attain a first-order description of Darmon points in our context, but this case it will be more difficult since we will require more commutative algebra and more complex first-order formulas. In general, whenever we work with a non-real number field $K$ and a regular algebraic function field $F$ in one variable over $K$, we will use the fact that $\mathcal{O}_{F/K}$ is a Dedekind domain to argue that for all $a_1,a_2,b\in F$ with $a_1a_2\left(1+4b\right)\neq 0$ the semi-local ring $T_{a_1,a_2,b,F/K}$, being a localization of a Dedekind domain, is itself a Dedekind domain; and since it has only finitely many prime ideals, it is a principal ideal domain; in particular a unique factorization domain.

Second, given $m,n,r,k,\ell\in\mathbb{Z}_{\geq 0}$, polynomials $p\in F\left[x_1,\cdots,x_m,y_1,\cdots,y_n,u_1,\cdots,u_k\right]$ and $q\in F\left[z_1,\cdots,z_r,v_1,\cdots,v_\ell\right]$, consider the formula\begin{equation}\label{disjEA-E}\left[\exists x_1\cdots \exists x_m\forall y_1\cdots \forall y_n\left(p\left(\vec{x},\vec{y},\vec{u}\right)=0\right)\right]\vee \left[\exists z_1\cdots\exists z_r\left(q\left(\vec{z},\vec{v}\right)\neq 0\right)\right],\end{equation}which is the general shape of a disjunction between an $\exists\forall$-formula and an $\exists$-formula. Such a disjunction is always equivalent to an $\exists\forall$-formula. Indeed, adding another variable $u$, \mbox{\cite[Lemma 1.2.3]{MR2297245}} says that there exists $W\in F\left[\vec{x},\vec{y},\vec{z},\vec{u},\vec{v},u\right]$ such that the system\[\begin{cases}p\left(\vec{x},\vec{y},\vec{u}\right)&=0,\\uq\left(\vec{z},\vec{v}\right)-1&=0\end{cases}\]is equivalent to $W\left(\vec{x},\vec{y},\vec{z},\vec{u},\vec{v},u\right)=0$. Formula \eqref{disjEA-E} is now equivalent to\[\exists x_1\cdots\exists x_m\exists z_1\cdots\exists z_r\forall y_1\cdots\forall y_n\forall u\left(W\left(\vec{x},\vec{y},\vec{z},\vec{u},\vec{v},u\right)\neq 0\right),\]as claimed.

Taking this into account, we now show:

\begin{theorem}\label{mainthmdarmon}Let $K$ be a number field, let $F$ be an algebraic function field in one variable over $K$, let $S$ be a finite subset of $\mathcal{V}\left(F/K\right)$, and let $n\in\mathbb{Z}_{\geq 1}$. Then the set $D_{F/K,S,n}$ is $\forall\exists\forall$-definable in $F$, uniformly with respect to all choices of $S$.
\end{theorem}

\begin{proof}As in the beginning of the proof of Theorem \ref{mainthmcampana}, there exists a finite extension $L$ of $K$ contained in the algebraic closure of $FK\left(\sqrt{-1}\right)$ such that $L$ is non-real and $F'\coloneqq FL$ is a regular algebraic function field in one variable over $L$. Also, $D_{F/K,S,n}=D_{F'/L,S',n}\cap F$ (where $S'$ is the set of all places in $\mathcal{V}\left(F'/L\right)$ lying over some place in $S$) by the same argument given in the last part of the proof of Theorem \ref{mainthmcampana}, so by Corollary \ref{backandforth} it suffices to show that $D_{F'/L,S',n}$ is $\forall\exists\forall$-definable in $F'$. By Proposition \ref{paramplaces3} there are $a_1,a_2,b,c,a_1',a_2',b',c'\in F'$ with $a_1a_2c\left(1+4b\right)a_1'a_2'c'\left(1+4b'\right)\neq 0$ and $S'=\Omega_{c,c'}^{a_1,a_2,b,a_1',a_2',b',F'/L}$. We claim that $D_{F'/L,S',n}$ is defined by the formula\begin{multline}\label{darmonfstorder}\forall \alpha_1\forall\alpha_2\forall \beta\forall\gamma\left[\begin{pmatrix}\alpha_1\alpha_2\gamma\left(1+4\beta\right)\neq 0\\\Omega_{c,c'}^{a_1,a_2,b,a_1',a_2',b',F'/L}\cap \Delta_{\gamma}^{\alpha_1,\alpha_2,\beta,F'/L}=\emptyset\\\Delta_{\alpha_1,\alpha_2,\beta,F'/L}=\Delta_{\gamma}^{\alpha_1,\alpha_2,\beta,F'/L}\end{pmatrix}\Rightarrow\right.\\\left.\exists y\exists z\left(y,z\in T_{\alpha_1,\alpha_2,\beta,F'/L}\wedge \operatorname{gcd}_{T_{\alpha_1,\alpha_2,\beta,F'/L}}\left(y,z\right)=1\wedge y=xz^n\right)\right].\end{multline}Observe that, in the context of the above formula, $\operatorname{gcd}_{T_{\alpha_1,\alpha_2,\beta,F'/L}}\left(y,z\right)=1$ is a meaningful and diophantine statement: indeed, $T_{\alpha_1,\alpha_2,\beta,F'/L}$ is a principal ideal domain, so the statement is equivalent to $\exists s\exists t\left(s,t\in T_{\alpha_1,\alpha_2,\beta,F'/L}\wedge sy+tz=1\right)$, which is existential by Proposition \ref{basicdefinitions}. We conclude that\begin{equation}\label{existential}\exists y\exists z\left(y,z\in T_{\alpha_1,\alpha_2,\beta,F'/L}\wedge \operatorname{gcd}_{T_{\alpha_1,\alpha_2,\beta,F'/L}}\left(y,z\right)=1\wedge y=xz^n\right)\end{equation}is an existential statement. It follows that, to prove that \eqref{darmonfstorder} is a $\forall\exists\forall$-formula, it suffices to show that\begin{equation}\label{forallexists}\begin{pmatrix}\alpha_1\alpha_2\gamma\left(1+4\beta\right)\neq 0\\\Omega_{c,c'}^{a_1,a_2,b,a_1',a_2',b',F'/L}\cap \Delta_{\gamma}^{\alpha_1,\alpha_2,\beta,F'/L}=\emptyset\\\Delta_{\alpha_1,\alpha_2,\beta,F'/L}=\Delta_{\gamma}^{\alpha_1,\alpha_2,\beta,F'/L}\end{pmatrix}\end{equation}is a $\forall\exists$-formula, since therefore its negation will be an $\exists\forall$-formula and then its disjunction with \eqref{existential} will be an $\exists\forall$-formula (as discussed prior to the proof of this theorem), thus rendering \eqref{darmonfstorder} as a $\forall\exists\forall$-formula (with exactly four initial universal quantifiers). And \eqref{forallexists} is indeed a $\forall\exists$-formula by Lemma \ref{emptyint2} and Lemma \ref{equalitytwodeltas}.

Given $x\in F'$ such that \eqref{darmonfstorder} holds on $x$ and $w\in\mathcal{V}\left(F'/L\right)\setminus S'$ such that \mbox{$w\left(x\right)\not\in\mathbb{Z}_{\geq 0}$}, we want to show that $n\mid w\left(x\right)$. Fixing $w'\in\mathcal{V}\left(F'/L\right)\setminus \left(S'\cup \left\{w\right\}\right)$, by Proposition \ref{paramplaces2} there exist $\alpha_1,\alpha_2,\beta,\gamma\in F$ with $\alpha_1\alpha_2\gamma\left(1+4\beta\right)\neq 0$ and $\Delta_{\alpha_1,\alpha_2,\beta,F'/L}=\Delta_{\gamma}^{\alpha_1,\alpha_2,\beta,F'/L}=\left\{w,w'\right\}$. Now formula \eqref{darmonfstorder} implies that $y=xz^n$ for some relatively prime $y,z\in \mathcal{O}_w\cap\mathcal{O}_{w'}$. Applying $w$ to this equality we get $w\left(y\right)\equiv w\left(x\right)\pmod{n}$, so it suffices to show that $w\left(y\right)=0$. Since $y,z\in \mathcal{O}_w\cap\mathcal{O}_{w'}$ are relatively prime, it suffices to show that $w\left(z\right)>0$. Indeed, if $w\left(z\right)=0$ then $w\left(y\right)=w\left(x\right)$, which is impossible because $w\left(x\right)<0\leq w\left(y\right)$.

Conversely, let $x\in D_{F'/L,S',n}$ and let us show that $x$ satisfies \eqref{darmonfstorder}. Fix $\alpha_1,\alpha_2,\beta,\gamma\in F'$ such that $\alpha_1\alpha_2\gamma\left(1+4\beta\right)\neq 0$, $\Omega_{c,c'}^{a_1,a_2,b,a_1',a_2',b',F'/L}\cap \Delta_{\gamma}^{\alpha_1,\alpha_2,\beta,F'/L}=\emptyset$, and \mbox{$\Delta_{\alpha_1,\alpha_2,\beta,F'/L}=\Delta_{\gamma}^{\alpha_1,\alpha_2,\beta,F'/L}$}. We want to show that $y=xz^n$ for some relatively prime $y,z\in T_{\alpha_1,\alpha_2,\beta,F'/L}$. Indeed, given \mbox{$w\in \Delta_{\gamma}^{\alpha_1,\alpha_2,\beta,F'/L}$} we have $w\not\in S'$, since\[S'\cap \Delta_{\gamma}^{\alpha_1,\alpha_2,\beta,F'/L}=\Omega_{c,c'}^{a_1,a_2,b,a_1',a_2',b',F'/L}\cap \Delta_{\gamma}^{\alpha_1,\alpha_2,\beta,F'/L}=\emptyset.\]Since $x\in D_{F'/L,S',n}$, we get $w\left(x\right)\in\mathbb{Z}_{\geq 0}\cup n\mathbb{Z}$ for all $w\in \Delta_{\gamma}^{\alpha_1,\alpha_2,\beta,F'/L}=\Delta_{\alpha_1,\alpha_2,\beta,F'/L}$.

Since $x\in F=\operatorname{Frac}\left(T_{\alpha_1,\alpha_2,\beta,F'/L}\right)$ and $T_{\alpha_1,\alpha_2,\beta,F'/L}$ is a unique factorization domain, then $x=\frac{y''}{y'}$ where $y',y''\in T_{\alpha_1,\alpha_2,\beta,F'/L}$ are relatively prime in $T_{\alpha_1,\alpha_2,\beta,F'/L}$. In this ring, the only irreducible elements are the generators of the maximal ideals induced by the finitely many places in $\Delta_{\gamma}^{\alpha_1,\alpha_2,\beta,F'/L}$. Moreover, given $w\in \Delta_{\gamma}^{\alpha_1,\alpha_2,\beta,F'/L}$ such that $w\left(y'\right)>0$, we must have $w\left(y''\right)=0$ and therefore $w\left(x\right)=-w\left(y'\right)<0$, which implies $n\mid w\left(x\right)$, and therefore $n\mid w\left(y'\right)$. Since $w$ was an arbitrary element of $\Delta_{\gamma}^{\alpha_1,\alpha_2,\beta,F'/L}$ with $w\left(y'\right)>0$, all the exponents in the factorization of $y'$ into products of irreducibles are divisible by $n$, which shows that $y'$ is, up to a unit in $T_{\alpha_1,\alpha_2,\beta,F'/L}$, a perfect $n$th power in $T_{\alpha_1,\alpha_2,\beta,F'/L}$, and therefore $x=\frac{y}{z^n}$ for some $y,z\in T_{\alpha_1,\alpha_2,\beta,F'/L}$ relatively prime. $\blacksquare$


\end{proof}

\begin{remark}\label{darmonuniform}With the same argument as in Remark \ref{campanauniform}, our proof shows that, in the context of Theorem \ref{mainthmdarmon}, we have actually proved that $D_{F/K,S,n}$ is \mbox{$\forall\exists\forall$-definable} in a uniform way with respect to all finite $S\subseteq\mathcal{V}\left(F/K\right)$.\end{remark}

Theorem \ref{mainthmdarmon} can be significantly improved when $S=\emptyset$. We show that, in this case, one attains a $\forall\exists$-definition.

\begin{theorem}\label{mainthmdarmon2}Let $K$ be a number field, let $F$ be an algebraic function field in one variable over $K$, and let $n\in\mathbb{Z}_{\geq 1}$. Then the set $D_{F/K,\emptyset,n}$ is $\forall\exists$-definable in $F$.
\end{theorem}

\begin{proof}As in the proof of Theorem \ref{mainthmdarmon}, we may assume that $K$ is non-real and $F/K$ is regular. We claim that $D_{F/K,\emptyset,n}=\left\{x\in F:v\left(x\right)\in\mathbb{Z}_{\geq 0}\cup n\mathbb{Z}\text{ for all $v\in\mathcal{V}\left(F/K\right)$}\right\}$ is defined by\begin{multline}\label{darmonfstorder2}\forall \alpha_1\forall\alpha_2\forall \beta\left[\left(\alpha_1\alpha_2\left(1+4\beta\right)\neq 0\right)\Rightarrow\right.\\\left.\exists y\exists z\left(y,z\in T_{\alpha_1,\alpha_2,\beta,F/K}\wedge \operatorname{gcd}_{T_{\alpha_1,\alpha_2,\beta,F/K}}\left(y,z\right)=1\wedge y=xz^n\right)\right].\end{multline}Indeed, fix $x\in F$ such that \eqref{darmonfstorder2} holds on $x$, and fix $v\in\mathcal{V}\left(F/K\right)$. We want to show that $v\left(x\right)\in\mathbb{Z}_{\geq 0}\cup n\mathbb{Z}$. Fix any $w\in\mathcal{V}\left(F/K\right)\setminus\left\{v\right\}$ and use Proposition \ref{paramplaces2} to find $\alpha_1,\alpha_2,\beta,\gamma\in F$ with $\alpha_1\alpha_2\left(1+4\beta\right)\neq 0$ and $\Delta_{\alpha_1,\alpha_2,\beta,F/K}=\left\{v,w\right\}$. It follows from \eqref{darmonfstorder2} that there exist relatively prime $y,z\in\mathcal{O}_v\cap\mathcal{O}_w$ with $y=xz^n$. Since $y,z\in\mathcal{O}_v\cap\mathcal{O}_w$ are relatively prime then $v\left(y\right)\geq 0$, $v\left(z\right)\geq 0$, and $v\left(y\right)v\left(z\right)=0$. And since $y=xz^n$ then $v\left(x\right)=v\left(y\right)-nv\left(z\right)$, thus $v\left(x\right)\in\mathbb{Z}_{\geq 0}\cup n\mathbb{Z}$.

Conversely, let $x\in D_{F/K,\emptyset,n}$, and let us show that \eqref{darmonfstorder2} holds on $x$. Given $\alpha_1,\alpha_2,\beta,\gamma\in F$ with $\alpha_1\alpha_2\left(1+4\beta\right)\neq 0$, we want to show that $y=xz^n$ for some relatively prime $y,z\in T_{\alpha_1,\alpha_2,\beta,F/K}$. Exactly as in the last paragraph of the proof of Theorem \ref{mainthmdarmon}, this is equivalent to $v\left(x\right)\in\mathbb{Z}_{\geq 0}\cup n\mathbb{Z}$ for each $v\in \Delta_{\alpha_1,\alpha_2,\beta,F/K}$, and this is true because $x\in D_{F/K,\emptyset,n}$. $\blacksquare$

\end{proof}

\printbibliography

\end{document}